\documentclass{amsart}

\usepackage{amssymb} \usepackage{amsfonts} \usepackage{amsmath}
\usepackage{amsthm} \usepackage{epsfig} 
\usepackage{color}
\usepackage{pinlabel}
\usepackage{amscd}
\usepackage[all]{xy}
\usepackage{pinlabel}
\usepackage{tcolorbox}

\newtheorem{lemma}{Lemma}%[section]
\newtheorem{teo}[lemma]{Theorem}
\newtheorem{prop}[lemma]{Proposition}
\newtheorem{cor}[lemma]{Corollary}

\theoremstyle{definition}

\newtheorem{example}[lemma]{Example}

\theoremstyle{remark}
\newtheorem{rem}[lemma]{Remark}

\newcommand{\Iso}{{\rm Isom}}

\newcommand{\matX}{\ensuremath {\mathbb{X}}}

\newcommand{\matR} {\ensuremath {\mathbb{R}}}

\newcommand{\matZ} {\ensuremath {\mathbb{Z}}}

\newcommand{\matH} {\ensuremath {\mathbb{H}}}
\newcommand{\matS} {\ensuremath {\mathbb{S}}}

\newcommand{\Vol} {\ensuremath {{\rm Vol}}}

\newcommand{\Isom} {\ensuremath {{\rm Isom}}}
\newcommand{\lk} {\ensuremath {{\rm link}}}

\author{Giovanni Italiano}
\address{Scuola Normale Superiore, Piazza dei Cavalieri, 7, 56126 Pisa, Italy}
\email{giovanni dot italiano at sns dot it}

\author{Bruno Martelli}
\address{Dipartimento di Matematica, Largo Pontecorvo 5, 56127 Pisa, Italy}
\email{bruno dot martelli at unipi dot it}

\author{Matteo Migliorini}
\address{Scuola Normale Superiore, Piazza dei Cavalieri, 7, 56126 Pisa, Italy}
\email{matteo dot migliorini at sns dot it}

\title[Hyperbolic manifolds that fiber algebraically]{Hyperbolic manifolds that fiber algebraically \\ up to dimension 8}

\begin{document}

\begin{abstract}
We construct some cusped finite-volume hyperbolic $n$-manifolds $M^n$ that fiber algebraically in all the dimensions $5\leq n \leq 8$. That is, there is a surjective homomorphism $\pi_1(M^n) \to \matZ$ with finitely generated kernel. 

The kernel is also finitely presented in the dimensions $n=7, 8$, and this leads to the first examples of hyperbolic $n$-manifolds $\widetilde M^n$ whose fundamental group is finitely presented but not of finite type. These $n$-manifolds $\widetilde M^n$ have infinitely many cusps of maximal rank and hence infinite Betti number $b_{n-1}$. They cover the finite-volume manifold $M^n$.

We obtain these examples by assigning some appropriate \emph{colours} and \emph{states} to a family of right-angled hyperbolic polytopes $P^5, \ldots, P^8$, and then applying some arguments of Jankiewicz -- Norin -- Wise \cite{JNW} and Bestvina -- Brady \cite{BB}. We exploit in an essential way the remarkable properties of the Gosset polytopes dual to $P^n$, and the algebra of integral octonions for the crucial dimensions $n=7,8$.
\end{abstract}

\maketitle

\section*{Introduction} \label{introduction:section}

We prove here the following theorem. Every hyperbolic manifold in this paper is tacitly assumed to be connected, complete, and orientable. 

\begin{teo} \label{main:teo}
In every dimension $5 \leq n \leq 8$ there are a finite volume hyperbolic $n$-manifold $M^n$ and a map $f \colon M^n \to S^1$ such that $f_* \colon \pi_1(M^n) \to \matZ$ is surjective with finitely generated kernel. The cover $\widetilde M^n = \matH^n/{\ker f_*}$ has infinitely many cusps of maximal rank.
When $n=7,8$ the kernel is also finitely presented. 
\end{teo}

We deduce in particular the following.

\begin{cor} \label{main:cor}
In dimension $n=7,8$ there is a hyperbolic $n$-manifold with finitely presented fundamental group and infinitely many cusps of maximal rank. The manifold covers a finite-volume hyperbolic manifold.
\end{cor}

The same assertion holds in the dimensions $n=5,6$ with ``finitely generated'' replacing ``finitely presented''.

For every $5\leq n \leq 8$ the group $\pi_1(M^n)$ is a finite-index subgroup of the arithmetic lattice ${\rm O}(n,1,\matZ)$.  
Recall that a group $\Gamma$ is \emph{of finite type}, denoted ${\rm F}$, if it has a finite classifying space, and \emph{of type ${\rm F}_m$} if it has a classifying space with finite $m$-skeleton. When $m=1$ or 2, being of type ${\rm F}_m$ is equivalent to $\Gamma$ being finitely generated or finitely presented, respectively.

\begin{cor} \label{alg:cor}
In dimension $n=7,8$ the lattice ${\rm O}(n,1,\matZ)$ contains a finitely presented subgroup $\Gamma$ without torsion and with infinite Betti number $b_{n-1}(\Gamma)$. In particular $\Gamma$ is ${\rm F}_2$ but not ${\rm F}_{n-1}$.
\end{cor}
\begin{proof}
Pick $\Gamma = \pi_1(\widetilde M^n) < \pi_1(M^n) < {\rm O}(n,1,\matZ)$. Since $\widetilde M^n$ has infinitely many cusps of maximal rank, it is homeomorphic to the interior of a manifold with infinitely many compact boundary components and hence has infinite Betti number $b_{n-1}(\widetilde M^n) = b_{n-1}(\Gamma)$.
\end{proof}

For every $5\leq n \leq 8$, all the finitely many cusps of $M^n$ are \emph{toric}, that is diffeomorphic to $T^{n-1}\times [0,+\infty)$, where we use $T^m$ to denote the $m$-torus. The cover $\widetilde M^n$ has infinitely many toric cusps, and finitely many cusps of rank $n-2$, each diffeomorphic to $T^{n-2} \times \matR \times [0, +\infty)$.

The manifolds $M^n$ and the maps $f$ are constructed explicitly and combinatorially, so some topological invariants may be calculated. The Euler characteristic, Betti numbers, and number of cusps of $M^n$ are listed in Table \ref{Mn:intro:table}. 

\begin{table}
\begin{center}
\begin{tabular}{c||c|ccccccc|c}
        & Euler & $b_1$ & $b_2$ & $b_3$ & $b_4$ & $b_5$ & $b_6$ & $b_7$ & Cusps\\
 \hline \hline
$M^5$ & 0 & 24 & 120 & 136 & 39 & 0 & 0 & 0 & 40  \\
$M^6$ & $-64$ & 18 & 183 & 411 & 207 & 26
 & 0 & 0 & 27 \\
$M^7$ & 0 & 
$182$ & $6321$ & $41300$ & $55139$ & $24010$ & $4031$ & 0 & 4032 \\
$M^8$ & $278528$ & 
\!365\! & \!33670\! & \!583290\! & \!1783226\! & \!1346030\! & \!456595\! & \!65279\! & \!65280\!
\end{tabular}
\vspace{.5 cm}
\caption{The Euler characteristic, Betti numbers, and the number of cusps of each hyperbolic $n$-manifold $M^n$. Each cusp of $M^n$ is toric, that is diffeomorphic to $T^{n-1} \times [0, +\infty)$.}
\label{Mn:intro:table}
\end{center}
\end{table}

\subsection*{Outline of the construction}
We use as building blocks a remarkable sequence of finite-volume right-angled polytopes $P^n \subset \matH^n$, defined for $3\leq n \leq 8$. The reflection group of $P^n$ is a finite-index subgroup of the integral lattice ${\rm O}(n, 1, \matZ)$. The polytope $P^n$ has both ideal and real vertices.

These polytopes made their first appearance in a paper of Agol, Long, and Reid \cite{ALR}. Their combinatorics was then described by Potyagailo and Vinberg \cite{PV}, and more information was later collected by Everitt, Ratcliffe, and Tschantz \cite{ERT}, who noticed in particular that $P^3, \ldots, P^8$ are combinatorially dual to the Euclidean \emph{Gosset polytopes} \cite{G} discovered by Gosset in 1900 and usually indicated with the symbols $-1_{21}, 0_{21}, \ldots, 4_{21}$. 

The Gosset polytopes form indeed a remarkable family of semi-regular polytopes. The 1-skeleton of $2_{21}$ is the configuration graph of the 27 lines in a generic cubic \cite{C}, while $3_{21}, 4_{21}$ are intimately connected with the integral octonions and the $E_8$ lattice. It has been a great pleasure to study these beautiful objects for this project.

The hyperbolic manifold $M^n$ is constructed by assembling some copies of $P^n$ by means of a suitable \emph{colouring} of its facets. This is a standard procedure that works with any right-angled polytope and was used (with a different language) by L\"obell in 1930 with the right-angled dodecahedron to build the first compact hyperbolic 3-manifold, see \cite{V}. For our purposes here it is important to find a colouring with few colours and many symmetries. Given the remarkable properties of the dual Gosset polytopes, it is natural to guess that some nice symmetric colourings for $P^n$ should exist, and we show here that this is indeed the case. In dimension $n=7,8$ we derive a natural colouring from the algebraic properties of the integral octonions.

Having constructed $M^n$, we build a map $f\colon M^n \to S^1$ by choosing an appropriate \emph{state} for $P^n$, that is a partition of its facets into two sets In and Out. A state defines a \emph{diagonal map} $f\colon M^n \to S^1$, as explained by Jankiewicz, Norin, and Wise \cite{JNW}. The homomorphism $f_*\colon \pi_1(M^n) \to \matZ$ is often non-trivial, and its kernel may be studied through Bestvina -- Brady theory of combinatorial Morse functions \cite{BB}. This fundamental paper furnishes in particular some conditions that, when satisfied, guarantee that $\ker f_*$ is finitely generated or, even better, finitely presented. The conditions are the following: if some simplicial complexes called \emph{ascending} and \emph{descending links} are all connected (respectively, simply connected), then the kernel is finitely generated (respectively, finitely presented).

The choice of an appropriate state for $P^n$ is crucial here, and we have used again the exceptional properties of the dual Gosset polytope, and of the integral octonions for $n=7,8$, to select a particularly symmetric state for which the above-mentioned conditions are satisfied. We took inspiration from a quaternions-generated state for the 24-cell that worked very well in \cite{BM} to design a similar octonions-generated state for $P^7$ and $P^8$ here.

After choosing the states, the conditions on the ascending and descending links have been verified by hand in the lower dimensions $n=3,4,5$ and with a computer code written in Sage in the higher dimensions $n=6,7,8$. The code may be downloaded from \cite{code}. 
The symmetries of the polytopes, of the colourings, and of the states have reduced considerably the computations involved, to keep them within few hours of CPU time. Without all these exceptional symmetries, not only the proof of Theorem \ref{main:teo}, but even the more straightforward calculation of the Betti numbers of $M^n$ would have been problematic, especially in the higher dimensions $n=7,8$ where the combinatorics is highly not trivial, as one can guess by looking at the size of the numbers in Table \ref{Mn:intro:table}. To the best of our knowledge the manifolds $M^7$ and $M^8$ are the first finite-volume hyperbolic manifolds in dimension $n\geq 7$ for which the Betti numbers have been computed. Some notable examples exist in the literature in dimension 5 and 6, see \cite{RT5, ERT}. 

The cover $\widetilde M^n = \matH^n/{\ker f_*}$ has a finitely generated fundamental group, and also a finitely presented one for $n=7,8$. It has infinitely many cusps for all $5\leq n \leq 8$ because $f$ is homotopically trivial on some cusp of $M^n$, which therefore lifts to infinitely many copies of itself in $\widetilde M^n$.

\subsection*{Related work}
We briefly discuss some works related to the present paper.

\subsubsection*{Coherence} The fundamental group of a hyperbolic $3$-manifold $M$ satisfies a number of nice properties, see \cite{AFW} for a widely comprehensive discussion. In particular, Scott proved \cite{Sc} that $\pi_1(M)$ is \emph{coherent}: every finitely generated subgroup is also finitely presented. 

This is not the case in higher dimensions, as first experienced by Kapovich and Potyagailo who constructed \cite{KP} in 1991 a geometrically finite hyperbolic 4-manifold with non-coherent fundamental group, see also \cite{P, KP2}. A compact example was then built by Bowditch and Mess \cite{BoMe} in 1994. Later on, Kapovich, Potyagailo, and Vinberg \cite{KPV} proved non-coherence for every non-uniform arithmetic lattice in $\Iso(\matH^n)$ with $n\geq 6$, and then Kapovich \cite{Kn} for every arithmetic hyperbolic lattice in dimension $n\geq 5, n\neq 7$. He also conjectured in \cite{Kn} that every hyperbolic lattice in dimension $n\geq 4$ should be non-coherent.

Corollaries \ref{main:cor} and \ref{alg:cor} describe an even wilder situation: in dimension $n=7,8$ there are finite-volume hyperbolic $n$-manifolds whose fundamental group contains subgroups that are ${\rm F}_2$ but not ${\rm F}_{n-1}$. The first example of a group that is ${\rm F}_2$ but not ${\rm F}_m$ for some $m\geq 3$ was provided by Stallings \cite{St}. It would be interesting to know if such subgroups may also occur in the intermediate dimensions $n=4,5,6$.

\subsubsection*{Algebraic fibrations}
Theorem \ref{main:teo} furnishes some explicit examples of algebraically fibering fundamental groups of hyperbolic manifolds.
We recall that a group $G$ \emph{fibers algebraically} if there is a surjective homomorphism $G \to \matZ$ with finitely generated kernel. 

When $G= \pi_1(M)$ is the fundamental group of a 3-manifold, by a well-known theorem of Stallings \cite{St} this condition is equivalent to the existence of a fibration $M\to S^1$. In higher dimensions this is false in general, and the first examples of algebraic fibrations on hyperbolic $n$-manifolds have appeared recently in dimension $n=4$ in \cite{AS, JNW}. The paper \cite{JNW} is the main inspiration for our work. In \cite{AS} the algebraic fibration is obtained by constructing a RFRS tower and then applying a recent general theorem of Kielak \cite{Kie} that transforms the RFRS property into an algebraic fibration (under some hypothesis). 

\subsubsection*{Perfect circle-valued Morse functions}
In dimension 4 the algebraic fibration can sometimes be promoted to a perfect circle-valued Morse function \cite{BM}. The algebraic fibrations constructed here in dimension $5\leq n \leq 8$ cannot be promoted to perfect circle-valued Morse functions because they are homotopically trivial on some cusps, see Section \ref{f:cusps:subsection}.
After writing a first draft of this paper, we could modify the construction in dimension $n=5$ to build a fibration \cite{IMM2}.

\subsubsection*{Infinitely many cusps}
Theorem \ref{main:teo} produces some hyperbolic manifolds with finitely presented fundamental group and infinitely many cusps. 

In dimension 3, every hyperbolic manifold with finitely generated fundamental group has only finitely many cusps. This is yet another nice property of 3-manifolds that fails in higher dimensions: we already know from \cite{Kf, KP} that there are some hyperbolic 4-manifolds with finitely generated fundamental group and infinitely many rank-1 cusps. With Theorem \ref{main:teo} we upgrade these examples by substituting ``rank-1'' with ``maximal rank'', and ``finitely generated'' with ``finitely presented''. The reader may consult \cite{K} for a comprehensive survey on 3-dimensional theorems that are not valid in higher dimension (the paper also contains a lot of interesting material). 

It is conjectured in \cite{Kf} that there is no hyperbolic $n$-manifold with finitely generated fundamental group and infinitely many cusps, all of maximal rank. We note that Theorem \ref{main:teo} does not disprove this conjecture, since $\widetilde M^n$ also contains finitely many cusps of rank $n-2$.

\subsection*{Structure of the paper}
We introduce the polytopes $P^n$ and construct the manifolds $M^n$ in Section \ref{M:section} by means of some appropriate colourings. Then in Section \ref{f:section} we introduce the techniques of \cite{JNW} and build the diagonal maps $f \colon M^n \to S^1$ via some carefully chosen states. By analysing the behaviour of $f$ we finally prove Theorem \ref{main:teo}.

\section{The manifolds $M^n$} \label{M:section}
We recall a general procedure to construct a manifold from a right-angled polytope $P$ by colouring its facets. This method was first introduced by Vesnin \cite{V} in 1987, inspired by the 1931 construction of L\"obell of the first known compact hyperbolic 3-manifold \cite{L} and by a paper of Al Jubouri \cite{A}. The method was then applied in dimension 4 by Bowditch and Mess \cite{BoMe}, and more recently by various authors, including Kolpakov -- Martelli \cite{KM} and Kolpakov -- Slavich \cite{KS}.

After recalling some general facts, we turn to the
polytopes $P^3,\ldots, P^8$ and choose some nice colouring to generate the manifolds $M^3, \ldots, M^8$. We will use the algebraic properties of the octonions to build $M^7$ and $M^8$.

\subsection{Colours} \label{colours:subsection}
Let $P\subset \matX^n$ be a right-angled finite polytope in some space $\matX^n = \matH^n, \matR^n$ or $\matS^n$. We always suppose that $P$ has finite volume. When $\matX^n = \matH^n$, the polytope $P$ may have both finite and ideal vertices. 
We can interpret $P$ as an orbifold $P=\matX^n /\Gamma$, where $\Gamma$ is the right-angled Coxeter group generated by the reflections $r_F$ along the facets $F$ of $P$. A presentation for $\Gamma$ is
$$\langle\ r_F\ |\ r_F^2, [r_F, r_{F'}]\ \rangle$$
where $F$ varies among the facets of $P$ and $F,F'$ among the pairs of adjacent facets.

 A \emph{$c$-colouring} of $P$ is the assignment of a colour (taken from some fixed set of $c$ elements) to each facet of $P$, such that incident facets have distinct colours. We generally use $\{1,\ldots, c\}$ as a palette of colours and suppose that every colour is painted on at least one facet.

Let $e_1,\ldots, e_c$ be the canonical basis of the $\matZ_2$-vector space $\matZ_2^c$. A colouring on $P$ induces a homomorphism $\Gamma \to \matZ_2^c$ that sends $r_F$ of to $e_j$, where $j$ is the colour of $F$. 
One verifies that the kernel $\Gamma' \triangleleft \Gamma$ acts freely on $\matX^n$, and hence we get a manifold $M=\matX^n/{\Gamma'}$ that orbifold-covers $P= \matX^n/\Gamma$ with degree $2^c$. 

\begin{rem} \label{general:rem}
A more general notion of colouring consists of assigning a vector $\lambda_F \in \matZ_2^c$ to each facet $F$ of $P$, that is not necessarily a member of the canonical basis. We require that facets with non-empty intersection are sent to independent vectors, see for instance \cite{KS}. We do not need this more general definition here.
\end{rem}

The manifold $M= \matX^n/\Gamma'$ is hyperbolic, flat, or elliptic, according to the model $\matX^n$, and is tessellated into $2^c$ copies of $P$. Geometrically, we may see $M$ as constructed by mirroring $P$ iteratively along facets sharing the same colours $1,\ldots, c$. %The manifold $M$ is canonically oriented by extending the orientation of $P$ along the mirrors. 

More precisely, we can describe the tessellation of $M$ into $2^c$ copies of $P$ as follows. For every vector $v\in \matZ_2^c$ we denote by $P_v$ an identical copy of $P$. We identify each facet $F$ of $P_v$ via the identity map with the same facet of $P_{v+e_j}$, where $j$ is the colour of $F$. This gives the tessellation of $M$.

We say that two colourings on $P$ are \emph{isomorphic} if they induce the same partition of facets, possibly after acting by some isometry of $P$. Isomorphic colourings yield isometric manifolds $M$.

As an example, we can always colour $P$ by assigning distinct colours to distinct facets. In this case $c$ equals the number of facets of $P$ and $\Gamma \to \matZ_2^c$ is just the abelianization homomorphism. With this choice, the resulting manifold $M$ can be quite big and often intractable (especially in higher dimension $n>3$), so it is often preferable to work with a small number of colours. Another fundamental reason for rejecting this inefficient colouring will be given below in Corollary \ref{trivial:cor}. 

Here are some more interesting examples:

\begin{figure}
 \begin{center}
  \includegraphics[width = 12.5 cm]{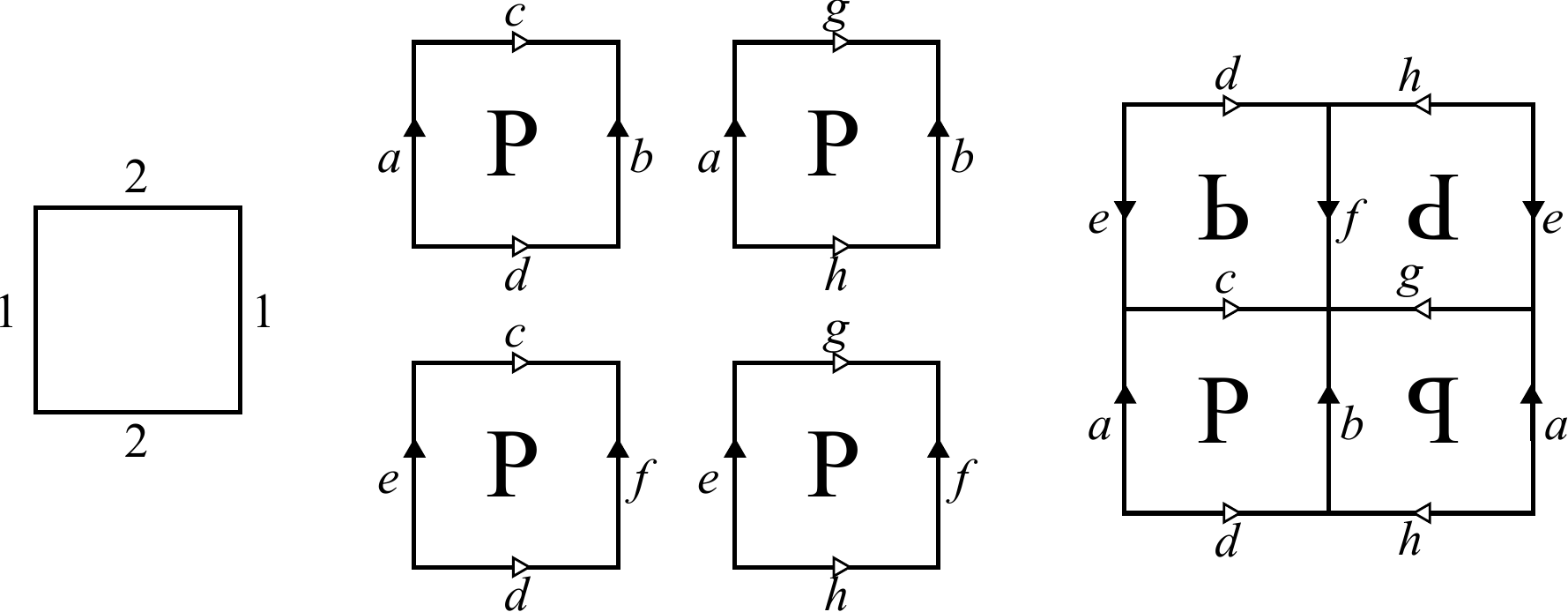}
 \end{center}
 \caption{A square $P$ with two colours (left). The flat manifold $M$ is constructed by taking four copies of $P$ and identifying the edges as shown (centre). We get a flat square torus (right).}
 \label{toro:fig}
\end{figure}

\begin{figure}
 \begin{center}
  \includegraphics[width = 3.2 cm]{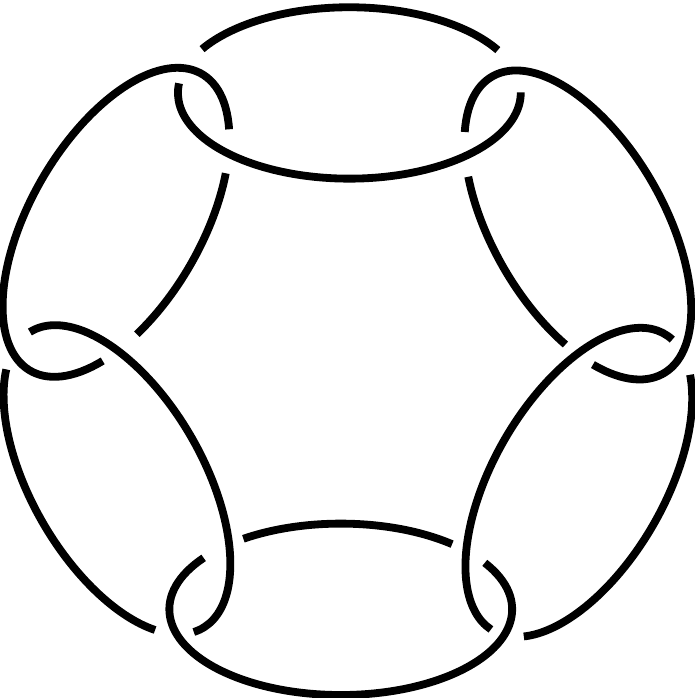}
 \end{center}
 \caption{The minimally twisted chain link with 6 components. }
 \label{chainlink:fig}
\end{figure}

\begin{itemize}
\item The Euclidean $n$-cube has a unique $n$-colouring up to isomorphisms, where opposite facets are coloured with the same colour. This colouring produces a flat torus tessellated into $2^n$ cubes. The case $n=2$ is shown in Figure \ref{toro:fig} and is easily generalised to any $n$. More generally, we will prove below that any colouring on the $n$-cube produces a flat torus.
\item The right-angled spherical $n$-simplex has a unique colouring up to isomorphisms: it has $n+1$ colours and it produces the spherical manifold $S^n$ with its standard tessellation into $2^{n+1}$ right-angled simplexes.
\item Every ideal hyperbolic polygon is right-angled in a vacuous sense (it has no finite vertices) and can be 1-coloured! Indeed the edges are pairwise disjoint and hence can all be given the same colour. The construction produces the double of the polygon, a hyperbolic punctured sphere. 
\item Every right-angled hyperbolic hexagon can be 2-coloured, and the result is the double of a geodesic pair-of-pants, that is a genus-2 hyperbolic surface, tessellated into four hexagons.
\item The ideal octahedron in $\matH^3$ has a unique 2-colouring up to isomorphisms. The colouring produces a cusped hyperbolic 3-manifold which is the complement of the minimally twisted chain link with 6 components shown in Figure \ref{chainlink:fig}, see \cite{KM} for more details.
\item The ideal 24-cell in $\matH^4$ has a unique 3-colouring, that produces a hyperbolic 4-manifold with 24 cusps with 3-torus sections, see \cite{KM}.
\end{itemize}

\begin{rem}
When $P$ is compact, it has some finite vertex incident to $n$ pairwise incident facets. These facets must have distinct colours and hence we necessarily have $c\geq n$. When $c=n$ the covering $M\to P$ has the minimum possible degree and is called a \emph{small cover}. These were studied in \cite{DJ}. Our examples will not be small covers because the polytopes that we consider have some ideal vertices, and moreover we will often have $c>n$.
\end{rem}

\begin{rem}
The manifold $M$ is always orientable: it suffices to orient $P_v$ like $P$ if and only if $v_1+\cdots +v_c$ is even (see an example in Figure \ref{toro:fig}). We note that $M$ is not guaranteed to be orientable if one uses the more general notion of colouring of Remark \ref{general:rem}. The crucial fact here is that all the vectors $e_j \in \matZ_2^c$ colouring the facets have an odd number of 1's in their entries. 
\end{rem}

\subsection{Cusp sections}
When $P\subset \matH^n$ has some ideal vertex, the resulting manifold $M$ has some cusps, and there is a simple and straightforward procedure to derive its shape directly from the combinatorics of $P$ and its colouring, that we now explain.

Let $v$ be an ideal vertex of $P\subset \matH^n$. The \emph{link} of $v$ in $P$ is by definition the intersection of $P$ with a small horosphere centered at $v$. It is a right-angled Euclidean $(n-1)$-parallelepiped $C$. We use the letter $C$ because a parallelepiped is combinatorially a cube, and in fact it will also be isometric to a cube in all the cases that are of interest here. 

The parallelepiped $C$ inherits a colouring from that of $P$: it suffices to assign to every $(n-2)$-facet of $C$ the colour of the $(n-1)$-facet of $P$ that contains it. 
The induced colouring on $C$ generates an abstract compact flat $(n-1)$-manifold $N$ that orbifold-covers $C$ by the procedure explained above. The manifold $N$ is tessellated into $2^{c'}$ copies of $C$, where $c'\leq c$ is the number of colours of $C$.

By construction, the preimage of $C$ in $M$ consists of some copies of $N$. The number of copies is equal to $2^h$, where $h=c-c'$ is the number of colours in $\{1,\ldots , c\}$ that are \emph{not} assigned to any facet incident to $v$, that is that are not assigned to any facet of $C$. The preimage of $C$ in $M$ consists of $2^c = 2^h \cdot 2^{c'}$ copies of $C$ in total.

Summing up: there are $2^h$ cusps in $M$ lying above $v$, each with section $N$ derived directly from $C$ and its induced colouring. Here are some examples:

\begin{itemize}
\item If $P\subset \matH^2$ is a 1-coloured ideal polygon, the link $C$ at each ideal vertex $v$ is a 1-coloured 1-cube (that is, a segment). Here $h=0$, the preimage of $C$ is a circle, and there is one cusp above each $v$. The punctured sphere $M$ has one cusp above each ideal vertex of $P$.
\item If $P\subset \matH^2$ is a 2-coloured ideal triangle, there are two types of ideal vertices. Two ideal vertices have a 2-coloured 1-cube as a link $C$, while the third ideal vertex has a 1-coloured 1-cube $C$. We have $h=0$ for the first two ideal vertices, and $h=1$ for the third. Therefore the counterimage of $C$ consists of one circle for each of the first two ideal vertices and two circles for the third. The manifold $M$ has 4 cusps overall, two above the first two vertices and two above the third. It is a four punctured sphere tessellated into four copies of $P$. See Figure \ref{triangles:fig}.
\item If $P$ is a 2-coloured ideal octahedron, it has 6 ideal vertices, and the link of each is a 2-coloured square $C$. We have $h=0$ on each ideal vertex, so the counterimage of $C$ in $M$ is a unique torus. The hyperbolic 3-manifold $M$ has 6 cusps overall, one above each ideal vertex of $P$. As already stated, $M$ is the complement of the link in Figure \ref{chainlink:fig}.
\item If $P$ is a 3-coloured ideal 24-cell in $\matH^4$, it has 24 ideal vertices, and the section of each is a 3-coloured 3-cube $C$, see \cite{KM}. We have $h=0$ on each ideal vertex, so the counterimage of $C$ is a single 3-torus. The hyperbolic 4-manifold $M$ has 24 toric cusps, one above each ideal vertex of $P$.
\end{itemize}

\begin{figure}
 \begin{center}
  \includegraphics[width = 7 cm]{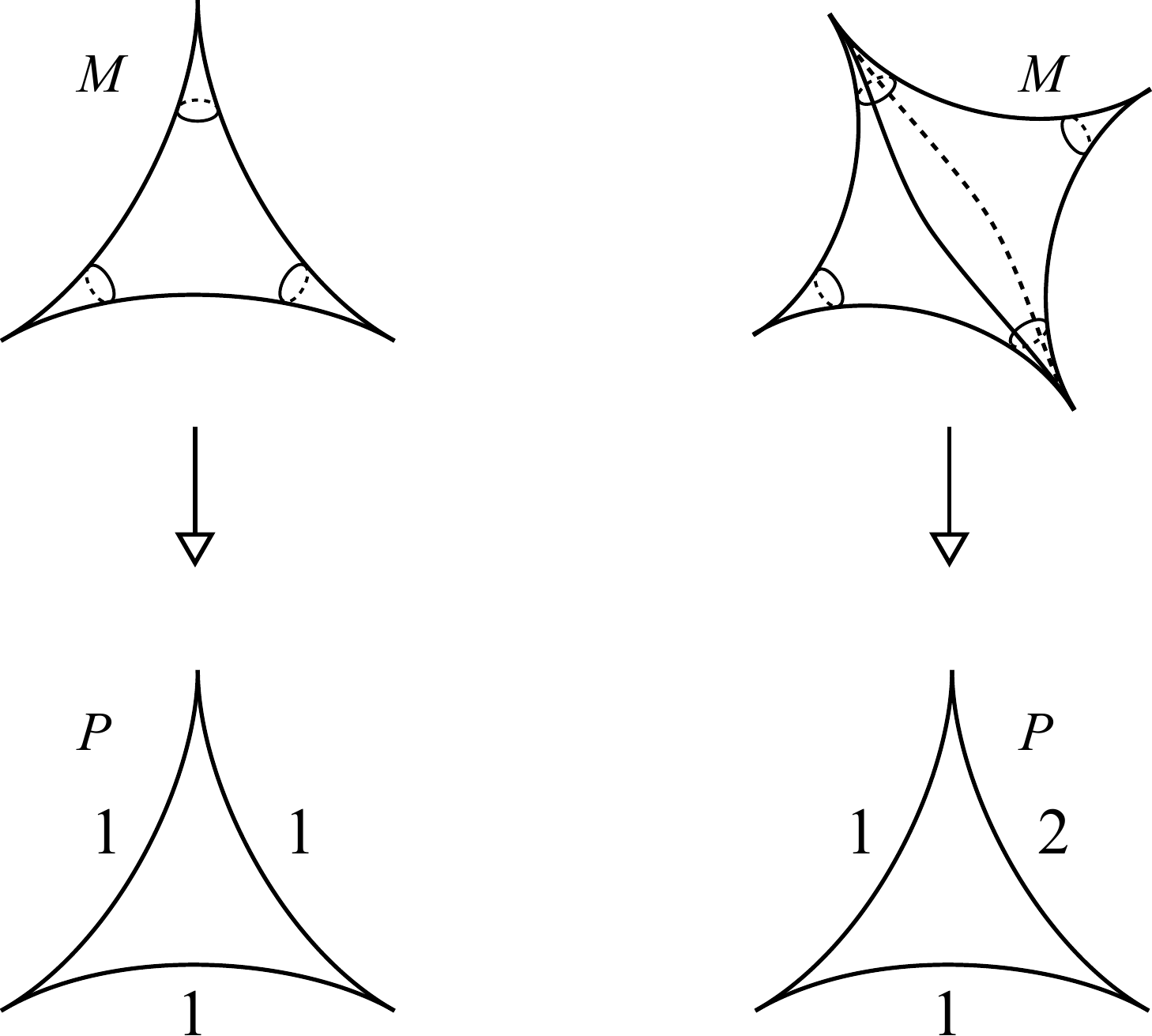}
 \end{center}
 \caption{When $P$ is an ideal triangle with 1 or 2 colours, the manifold $M$ is a sphere with 3 or 4 punctures, respectively.}
 \label{triangles:fig}
\end{figure}

\subsection{The Euclidean parallelepiped}
One basic example is the Euclidean right-angled $n$-parallelepiped
$$C = [0,l_1] \times \cdots \times [0,l_n] \subset \matR^n.$$

Fix a $c$-colouring of $C$. Only opposite facets are disjoint and hence may share the same colour. Therefore we have $n \leq c \leq 2n$, there are $2n-c$ pairs of opposite facets with the same colour and the remaining $2(c-n)$ facets with distinct colours. Let $M$ be the flat manifold produced by the $c$-colouring of $C$.

\begin{prop} \label{cubes:prop}
The resulting flat manifold $M$ is a $n$-torus isometric to a product of circles of lengths $a_1l_1, \ldots, a_nl_n$. Here $a_i$ equals 2 or 4 depending on whether the $i$-th pair of opposite facets share the same colour or not.
\end{prop}

%The $n$-torus $M$ is tessellated into $2^{2n-c} \cdot 4^{c-n} = 2^c$ copies of $C$.

\begin{proof}
Recall that $M = \matR^n/\Gamma'$ where $\Gamma$ is the reflection group of $C$ and $\Gamma'\triangleleft \Gamma$ is the kernel of the map $\Gamma \to \matZ_2^c$ induced by the colouring.

Let $r_{i,1}$ and $r_{i,2}$ be the reflections along the opposite facets of $C$ that are orthogonal to the $i$-th axis, for $i=1,\ldots, n$. The composition $r_{i,1}r_{i,2}$ is a translation along the axis of distance $2l_i$. If the facets share the same colour we have $r_{i,1}r_{i,2} \in \Gamma'$, while if they do not we have $r_{i,1}r_{i,2}r_{i,1}r_{i,2} \in \Gamma'$. This shows that
$$a_1l_1 \matZ \times \cdots \times a_nl_n \matZ < \Gamma'$$
where $a_i$ equals 2 or 4 depending on whether the $i$-th pair of opposite facets share the same colour. These two subgroups have the same index in $\Gamma$ since
$$2^{2n-c} \cdot 4^{c-n} = 2^c.$$ 
Therefore $\Gamma' = a_1l_1 \matZ \times \cdots \times a_nl_n \matZ$
and $M$ is as stated.
\end{proof}

The proof also shows that $M$ is tessellated into $2^{2n-c} \cdot 4^{c-n} = 2^c$ copies of $C$. The two extreme cases are the following: if $c=n$ then $M$ is tessellated into $2^c$ copies of $C$, while if $c=2n$ then $M$ is tessellated into $4^c$ copies. 

A cusp in a hyperbolic $n$-manifold is \emph{toric} if its section is a flat $(n-1)$-torus. We summarize our discussion as follows.

\begin{cor} \label{cusps:cor}
If $P\subset \matH^n$ is right-angled with some ideal vertices, every colouring on $P$ produces some hyperbolic $n$-manifold $M$ whose cusps are all toric.

If $P$ has $c$ colours and $v$ is an ideal vertex, there are $2^{c-c'}$ toric cusps in $M$ above $v$, where $c'$ is the number of distinctly coloured facets incident to $v$.
\end{cor}

\begin{rem}
If we use the more general notion of colouring of Remark \ref{general:rem}, non-toric cusps may also appear: see for instance \cite{FKS}. 
\end{rem}

\subsection{A program in Sage}
We have written a general program in Sage, available from \cite{code}, that may be used to study a coloured right-angled polytope $P$ and the resulting manifold $M$. The program takes as an input the incidence graph of the facets of $P$ and their colouring, and produces as an output some information on $P$ and more importantly on $M$. It calculates in particular the Betti numbers of $M$ via the formula stated in \cite[Theorem 1.1] {CP}, also explained in \cite[Section 2.2]{FKS}, and the number of cusps of $M$ using Corollary \ref{cusps:cor}.

\subsection{The right-angled hyperbolic polytopes}
We refer to the excellent papers \cite{ERT} and \cite{PV} for an introduction to the sequence of right-angled hyperbolic polytopes $P^3,\ldots, P^8$. These have many beautiful properties that we now briefly summarize.

\begin{table}
\begin{center}
\begin{tabular}{c||ccccccc}
        & Facets & Ideal & Finite & $\Isom(P^n)$ & Order & Dual\\
 \hline \hline
$P^3$ & 6 & 3 & 2 & $A_1 \times A_2$ & 12 & Triangular prism \\
$P^4$ & 10 & 5 & 5 & $A_4$ & 120 & Gosset $0_{21}$ \\
$P^5$ & 16 & 10 & 16 & $D_5$ & 1920 & Gosset $1_{21}$ \\
$P^6$ & 27 & 27 & 72 & $E_6$ & 51840 & Gosset $2_{21}$ \\
$P^7$ & 56 & 126 & 576 & $E_7$ & 2903040 & Gosset $3_{21}$ \\
$P^8$ & 240 & 2160 & 17280 & $E_8$ & 696729600 & Gosset $4_{21}$
\end{tabular}
\vspace{.5 cm}
\caption{The number of facets, ideal vertices, and finite vertices of $P^n$, the isometry group $\Isom(P^n)$ expressed as a Weyl group and its order $|\Isom(P^n)|$, and the dual Euclidean polytope.}
\label{Pn:table}
\end{center}
\end{table}

Each $P^n \subset \matH^n$ is a finite volume right-angled polytope with both finite and ideal vertices. The link of a finite or ideal vertex is respectively a right-angled spherical $(n-1)$-simplex and a Euclidean $(n-1)$-cube.
The numbers of facets, ideal vertices, and finite vertices of $P^n$ are listed in Table \ref{Pn:table}, together with the isometry group of $P^n$ and its order. 
The isometry group acts transitively on the facets, so in particular these are all isometric: in fact, every facet of $P^n$ is isometric to $P^{n-1}$ when $n\geq 4$. The quotient of $P^n$ by its isometry group is a simplex.

\subsection{Euler characteristic}
Recall that the orbifold Euler characteristic of a hyperbolic right-angled polyhedron $P$ is zero in odd dimension while in even dimension it can be calculated via the simple formula
$$\chi(P) = \sum_{i=0}^n (-1)^i \frac{f_i}{2^{n-i}}$$
where $f_i$ is the number of $i$-dimensional faces of $P$, including $P$ itself (so $f_n=1$). Only real vertices (not the ideal ones) contribute to $f_0$. 
From this formula we deduce the well-known \cite{ERT} values 
$$\chi(P^4) = 1/16, \qquad \chi(P^6) = -1/8, \qquad \chi(P^8) = 17/2.$$

In even dimension the Euler characteristic and the volume are roughly the same thing, up to a constant that will be recalled below.

\subsection{The dual Gosset polytopes}
Combinatorially, the polytopes $P^n$ are dual to the \emph{Gosset polytopes} listed in the last column of Table \ref{Pn:table} and discovered by Gosset \cite{G} in 1900, see \cite{ERT}. Every Gosset polytope is a Euclidean polytope with regular facets, whose isometry group (which is the same as $\Isom(P^n)$) acts transitively on the vertices. The regular facets of the Gosset polytope are of two types: some $(n-1)$-simplexes (dual to the real vertices of $P^n$) and some $(n-1)$-octahedra (dual to the ideal vertices of $P^n$). A \emph{$k$-octahedron} here is the regular polytope dual to the $k$-cube (sometimes also called $k$-orthoplex).

We will describe a colouring of $P^n$ as a colouring of the vertices of the dual Gosset polytope, where we require of course that two vertices adjacent connected by an edge must have distinct colours (so only the 1-skeleton of the dual Gosset polytope is important at this stage). We would like to find some colouring with a reasonably small number of colours, and possibly a high degree of symmetry: we are confident that some natural choices should arise from the exceptional properties of $P^n$ and their dual Gosset polytopes, and this is indeed the case as we will see.

We now analyse the polyhedra $P^3, \ldots, P^8$ individually. For each $P^n$ we define a colouring and study the resulting hyperbolic manifold $M^n$.

\subsection{The manifold $M^3$} The hyperbolic polyhedron $P^3\subset \matH^3$ is the right-angled double pyramid with triangular base shown in Figure \ref{bipyramid:fig}. The three vertices of the triangular base are ideal, while the two remaining vertices are real. Each of the 6 faces $F$ of $P^3$ is a triangle with a right-angled real vertex and two ideal vertices.

\begin{figure}
 \begin{center}
  \includegraphics[width = 5 cm]{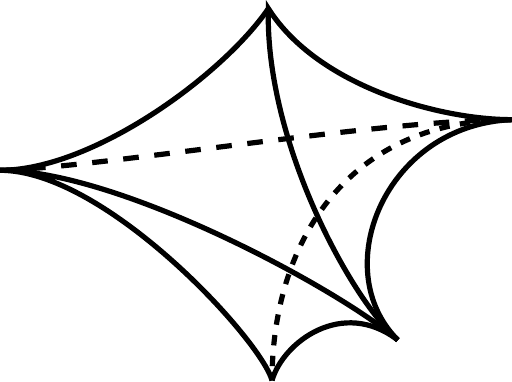}
 \end{center}
 \caption{The polyhedron $P^3$ is a right-angled bipyramid with three ideal vertices along the horizontal plane and two real ones (top and bottom in the figure).}
 \label{bipyramid:fig}
\end{figure}

\begin{figure}
 \begin{center}
  \includegraphics[width = 4.5 cm]{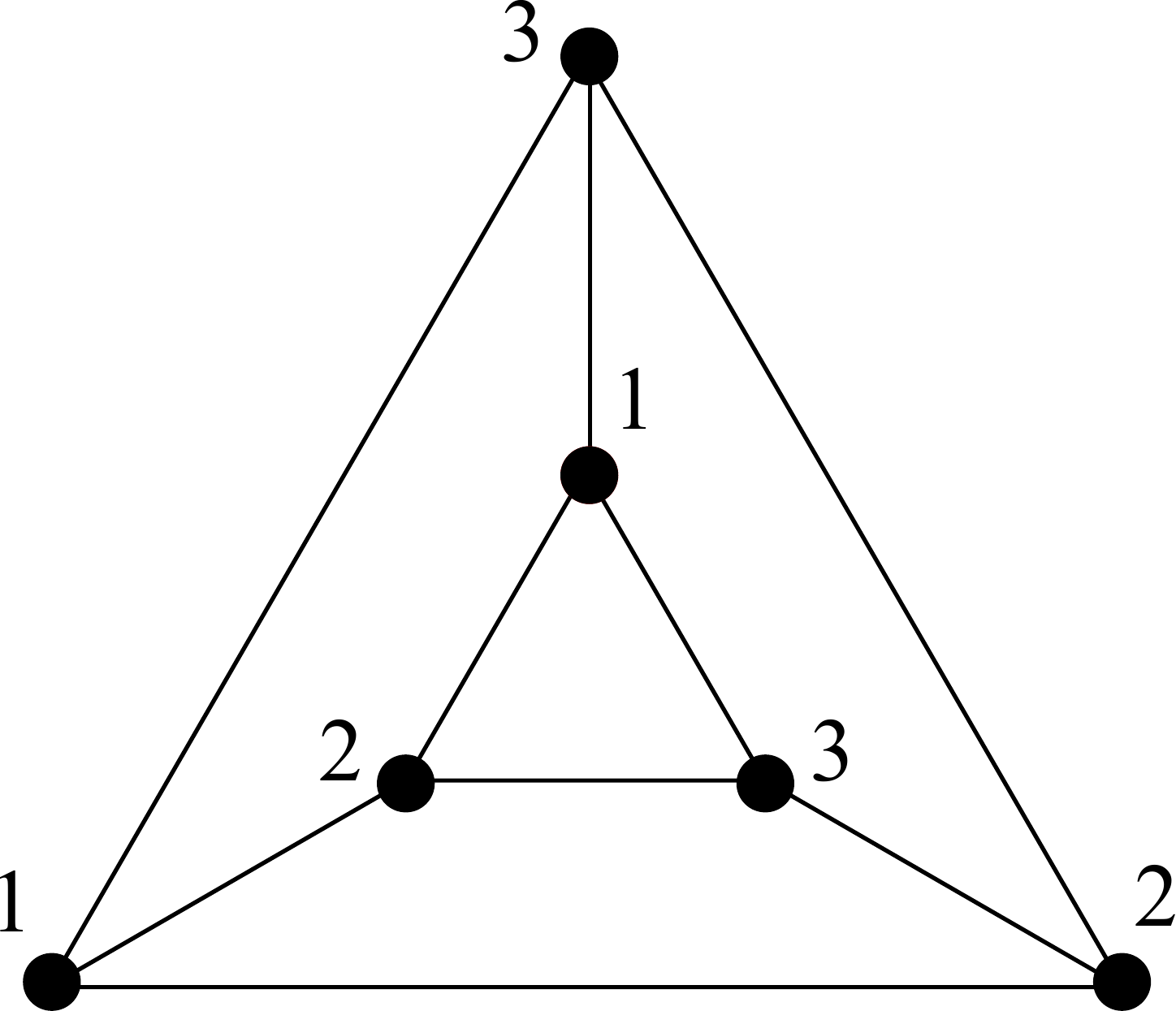}
 \end{center}
 \caption{The 1-skeleton of the triangular prism has a unique 3-colouring up to isomorphism, shown here.}
 \label{P3:fig}
\end{figure}

The dual Gosset polytope is a triangular prism, whose faces are two base equilateral triangles and three lateral squares. Its 1-skeleton is shown schematically in Figure \ref{P3:fig}. It can be coloured with 3 colours in a unique way (shown in the figure) up to isomorphism. Therefore $P^3$ has a unique 3-colouring up to isomorphism. The polyhedron cannot be coloured with less than 3 colours.

We equip $P^3$ with this 3-colouring. This produces a hyperbolic 3-manifold $M^3$, tessellated into $2^3=8$ copies of $P^3$. %Every face $F$ of $P^3$ inherits a 2-colouring and lifts to a thrice-punctured sphere geodesically embedded in $M^3$, tessellated into 4 copies of $F$.

The link of each ideal vertex $v$ of $P^3$ is a square $C$, that is dual to a square face of the Gosset prism. We see from Figure \ref{P3:fig} that $C$ is 3-coloured: two opposite edges of $C$ have distinct colours, and the other two opposite edges have the same colour. By Corollary \ref{cusps:cor} the counterimage of $C$ consists of a single (because $2^{3-3}=1$) torus cusp section in $M^3$. The hyperbolic manifold $M^3$ has therefore 3 cusps, one above each vertex $v$ of $P$. 

Using Sage we have calculated the Betti numbers of $M^3$:
$$b_0 = 1, \qquad b_1 = 3, \qquad b_2 = 2.$$

We get of course $\chi(M^3)=0$. 

\subsection{The manifold $M^4$}
The hyperbolic polytope $P^4\subset \matH^4$ is fully described in \cite{PV, RT4} and we refer to these sources for more details. It has 10 facets, each isometric to $P^3$. It has also 5 real vertices and 5 ideal vertices.

\begin{figure}
 \begin{center}
  \includegraphics[width = 12 cm]{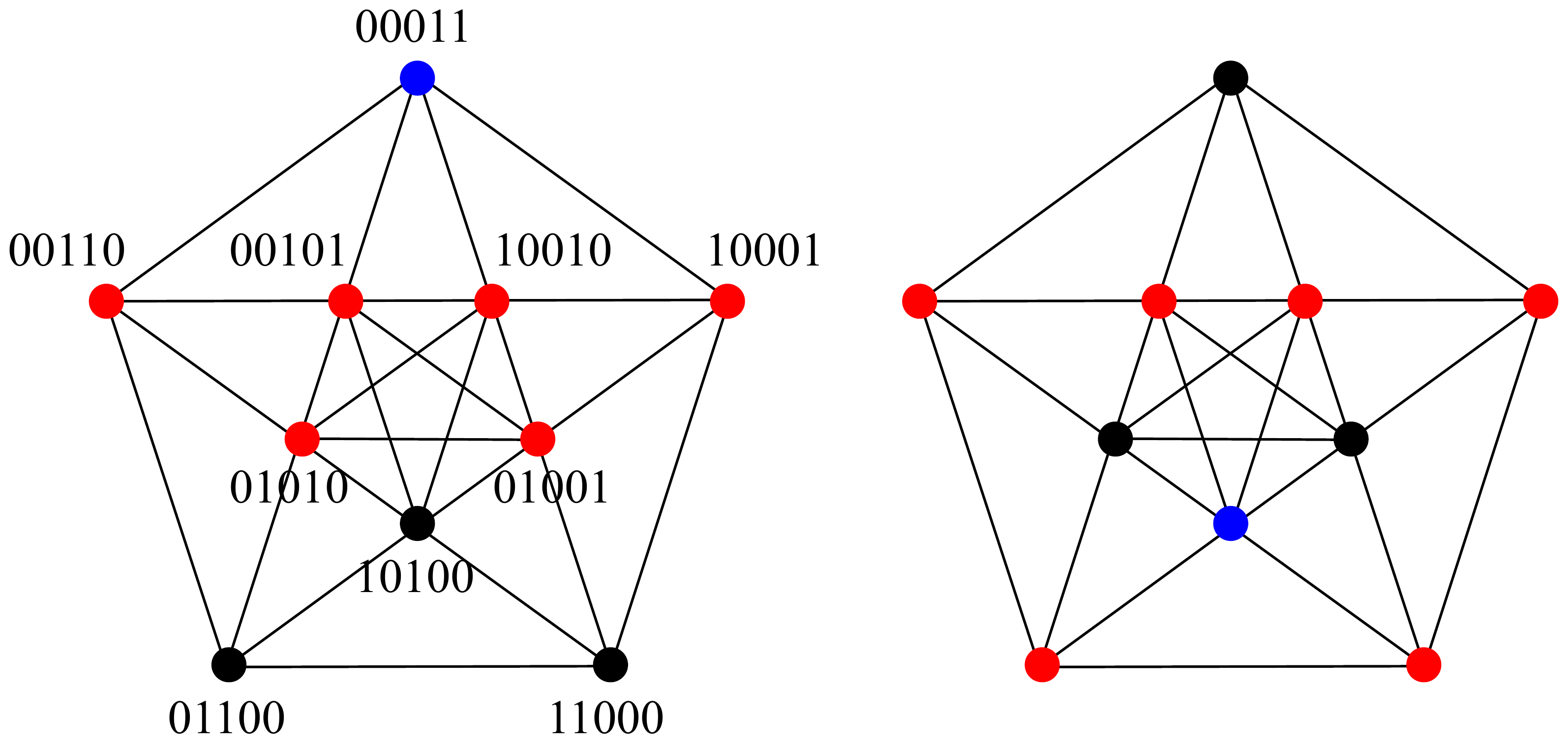}
 \end{center}
 \caption{The orthogonal projection of the 1-skeleton of the rectified simplex $0_{21}$ on the plane in $\matR^5$ generated by $(1, \epsilon, -\epsilon, -1, 0)$ and its cyclic permutations, where $\epsilon =(\sqrt 5 -1)/2 = 2\cos (2\pi/5)$ is the positive root of $\epsilon^2 + \epsilon -1$. The image of the vertex $(0,0,0,1,1)$ is indicated as $00011$, and so on. Some edges are superposed along the projection, so two vertices that are connected by an edge on the plane projection may not be so in $0_{21}$. To clarify this ambiguity we have chosen a blue vertex and painted in red the 6 vertices adjacent to it, in two cases (all the other cases are obtained by rotation).}
 \label{P4:fig}
\end{figure}

\begin{figure}
 \begin{center}
  \includegraphics[width = 5 cm]{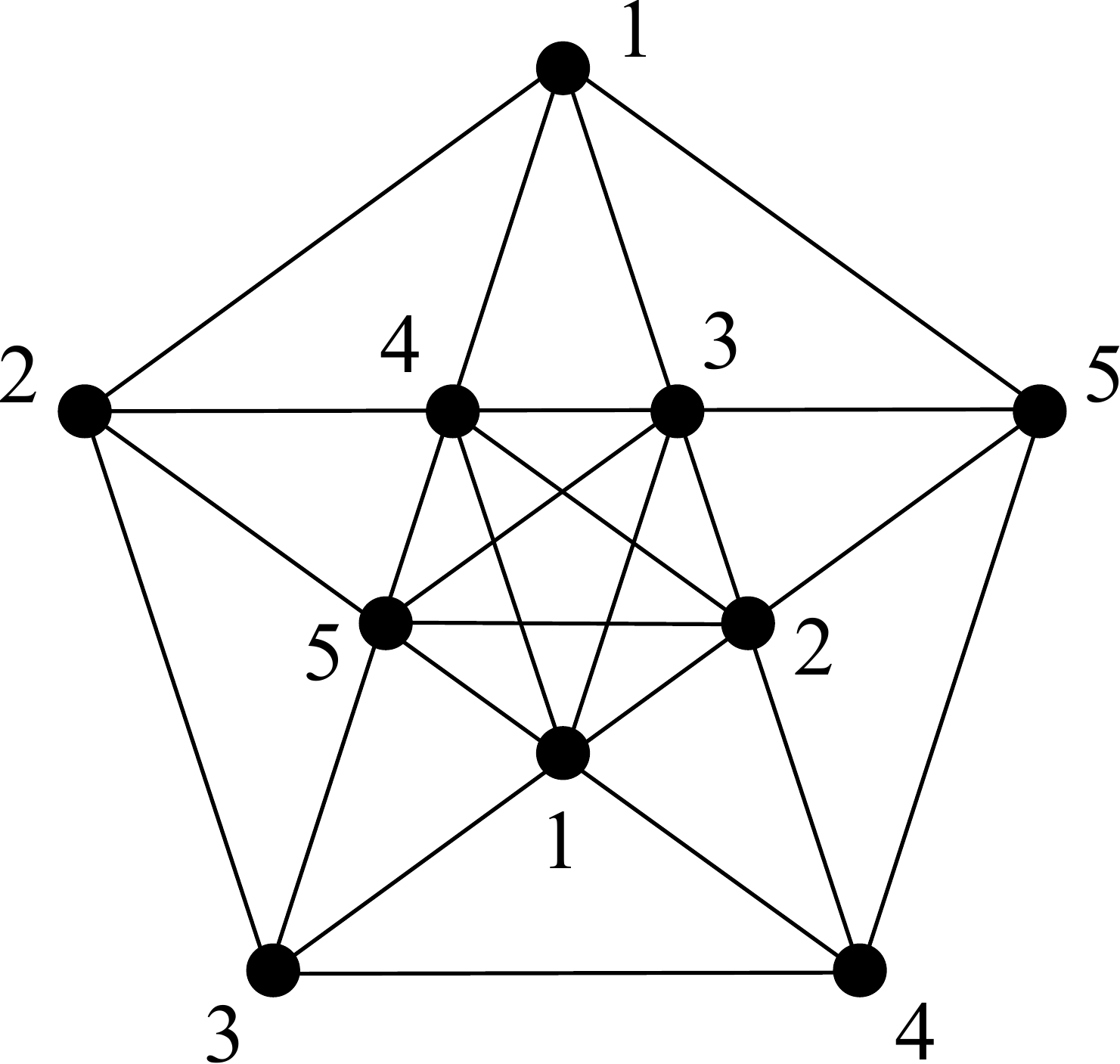}
 \end{center}
 \caption{A 5-colouring of the 1-skeleton of $0_{21}$ and hence of $P^4$.}
 \label{P4_colour:fig}
\end{figure}

The dual Gosset polytope $0_{21}$ is the 4-dimensional rectified simplex. That is, it is the convex hull of the midpoints of the 10 edges of a regular 4-dimensional simplex. Its 10 vertices may be seen in $\matR^5$ as the points obtained by permuting the coordinates of $(0,0,0,1,1)$. Two such vertices are adjacent if they differ only in two coordinates. The Gosset polytope $0_{21}$ has 10 facets; of these, 5 are regular tetrahedra (created by the rectification) dual to the finite vertices of $P^4$, and 5 are regular octahedra (the rectified facets of the original regular 4-simplex) dual to the ideal vertices of $P^4$.

A convenient orthogonal plane projection of the 1-skeleton of $0_{21}$ is shown in Figure \ref{P4:fig}. We assign to $0_{21}$, and hence to $P^4$, the 5-colouring depicted in Figure \ref{P4_colour:fig}. This produces a hyperbolic 4-manifold $M^4$, tessellated into $2^5 = 32$ copies of $P^4$. We have $\chi(M^4) = 32/16 = 2$.

\begin{figure}
 \begin{center}
  \includegraphics[width = 3.5 cm]{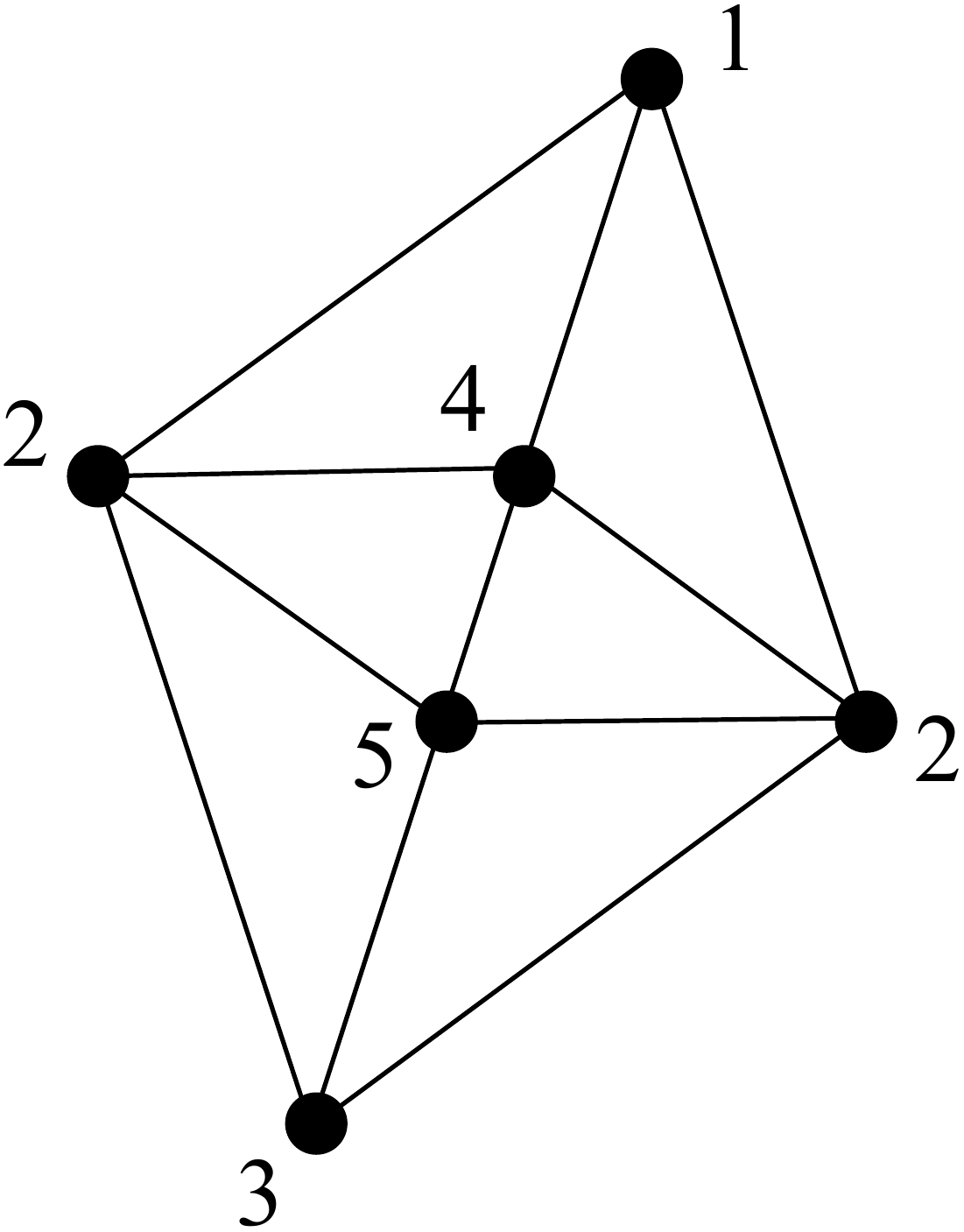}
 \end{center}
 \caption{An octahedral facet of $P^4$. This is a subgraph of the 1-skeleton in Figure \ref{P4_colour:fig}. Some edges are superposed.}
 \label{P4_facet:fig}
\end{figure}

The polytope $P^4$ has 5 ideal vertices $v_1,\ldots, v_5$. Each $v_i$ is dual to the octahedral facet of $0_{21}$ contained in the coordinate hyperplane $x_i=0$, whose 6 vertices in Figure \ref{P4:fig} are precisely those with $x_i=0$. The case $i=1$ is shown in Figure \ref{P4_facet:fig}. We can see on the figure that the octahedron is 5-coloured. The other four octahedra are obtained from this one by rotating the plane projection diagram and they are also 5-coloured.

We have discovered that the link of each ideal vertex of $P^4$ is a 5-coloured cube $C$. By Corollary \ref{cusps:cor} the counterimage of $C$ consists of a single (since $2^{5-5}=1$) toric cusp section in $M^4$. Therefore the hyperbolic manifold $M^4$ has 5 cusps overall, one above each ideal vertex of $P^4$.

Using Sage we have calculated the Betti numbers of $M^4$:
$$b_0 = 1, \qquad b_1 = 5, \qquad b_2 = 10, \qquad b_3 = 4.$$

We get $\chi(M^4)=2$ again.

\subsection{The manifold $M^5$} 
The hyperbolic polytope $P^5\subset \matH^5$ is fully described in \cite{PV, RT5} and we refer to these sources for more details. It has 16 facets, each isometric to $P^4$. It also has 16 real vertices and 10 ideal vertices. Every real vertex is opposed to a facet.

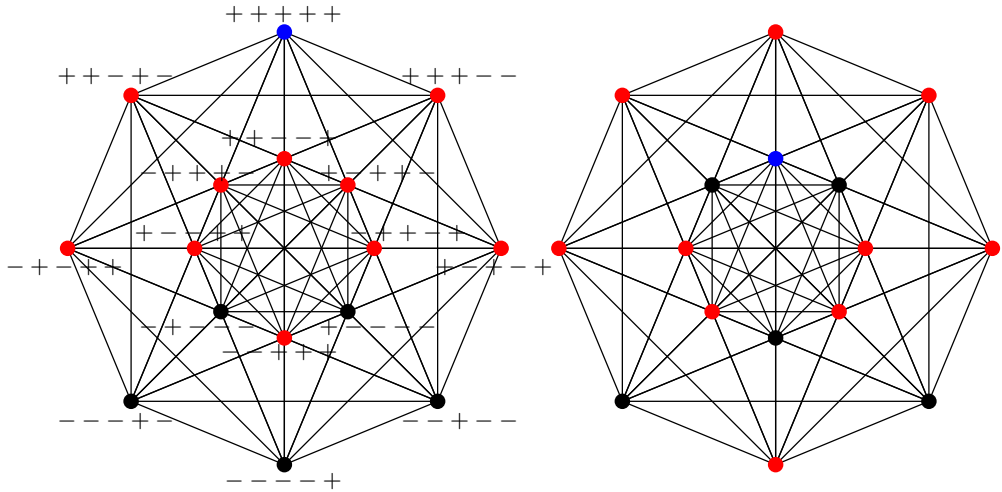
\begin{figure}
\labellist
\small\hair 2pt
\pinlabel $+++++$ at 260 535
\pinlabel $----+$ at 260 -10
\pinlabel $---+-$ at 65 60
\pinlabel $--+--$ at 465 60
\pinlabel $-+---$ at 160 170
\pinlabel $+----$ at 370 170
\pinlabel $--+++$ at 255 140

\pinlabel $++-+-$ at 65 463
\pinlabel $+++--$ at 465 463
\pinlabel $-+++-$ at 160 350
\pinlabel $+-++-$ at 370 350
\pinlabel $++--+$ at 255 390

\pinlabel $+-+-+$ at 505 240
\pinlabel $-++-+$ at 405 280
\pinlabel $+--++$ at 155 280
\pinlabel $-+-++$ at 5 240

\endlabellist
 \begin{center}
  \includegraphics[width = 12.5 cm]{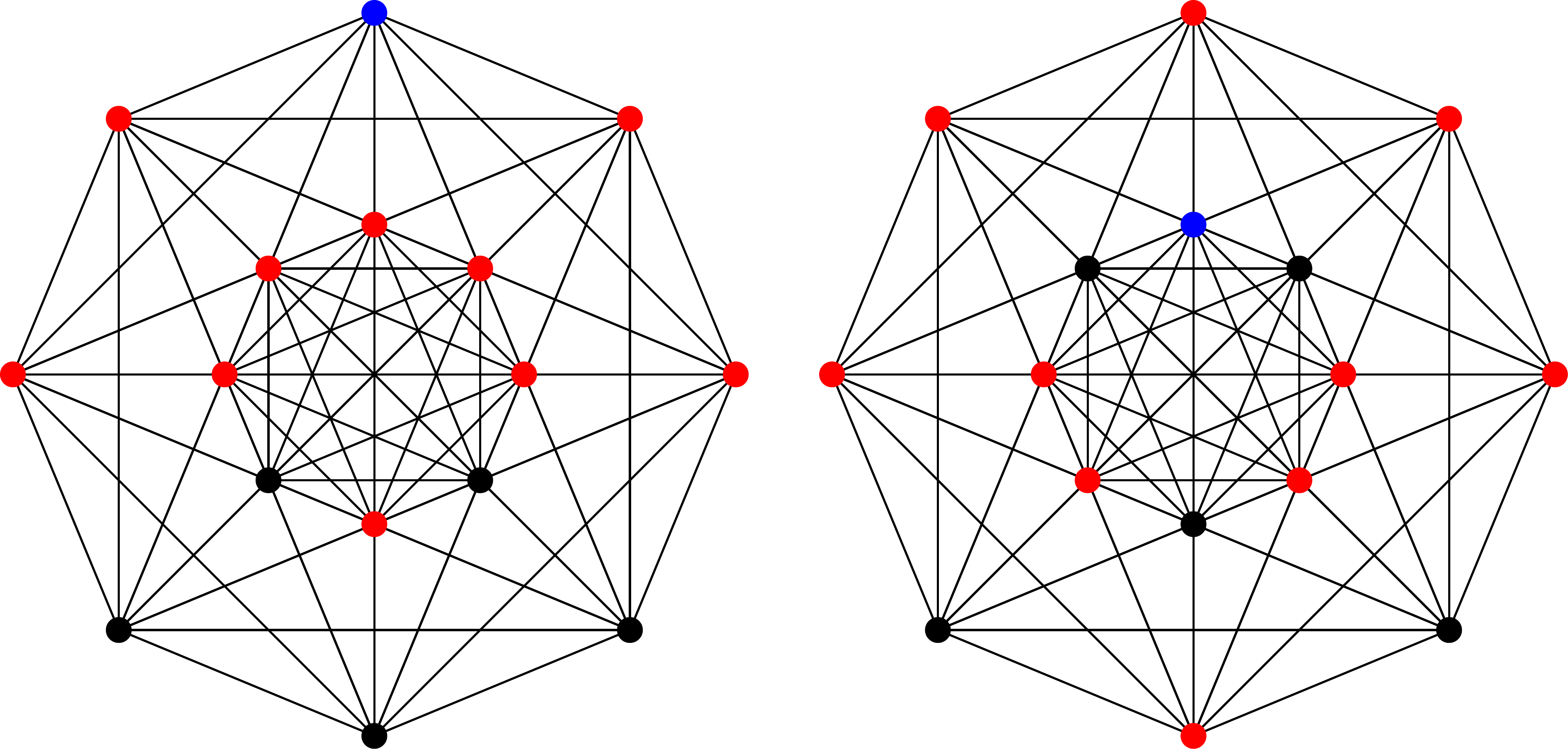}
 \end{center}
 \caption{The orthogonal projection of the 1-skeleton of $1_{21}$ on the plane spanned by the vectors $(\sqrt 2, \sqrt 2, 2-\sqrt 2, 2-\sqrt 2, 0)$ and $(2-\sqrt 2, \sqrt 2-2, \sqrt 2, -\sqrt 2, 0)$. The string $\pm \pm \pm \pm \pm$ indicates the projection of the vertex $(\pm 1, \pm 1, \pm 1, \pm 1, \pm 1)$. Some edges are superposed along the projection, so two vertices that are connected by an edge on the plane projection may not be so in $1_{21}$. To clarify visually this ambiguity we have chosen a blue vertex and painted in red the 10 vertices adjacent to it, in two cases (all the other cases are obtained by rotation).}
 \label{P5:fig}
\end{figure}

The dual Gosset polytope $1_{21}$ has 16 vertices. We can represent these in $\matR^5$ as the vertices $(\pm 1, \pm 1, \pm 1, \pm 1, \pm 1)$ with an odd number of minus signs. Two vertices are connected by an edge if they differ only in two coordinates. The Gosset polytope $1_{21}$ has 26 facets; of these, 16 are regular 4-simplexes dual to the finite vertices of $P^5$, and 10 are regular 4-octahedra dual to the ideal vertices of $P^5$. 

\begin{figure}
 \begin{center}
  \includegraphics[width = 7 cm]{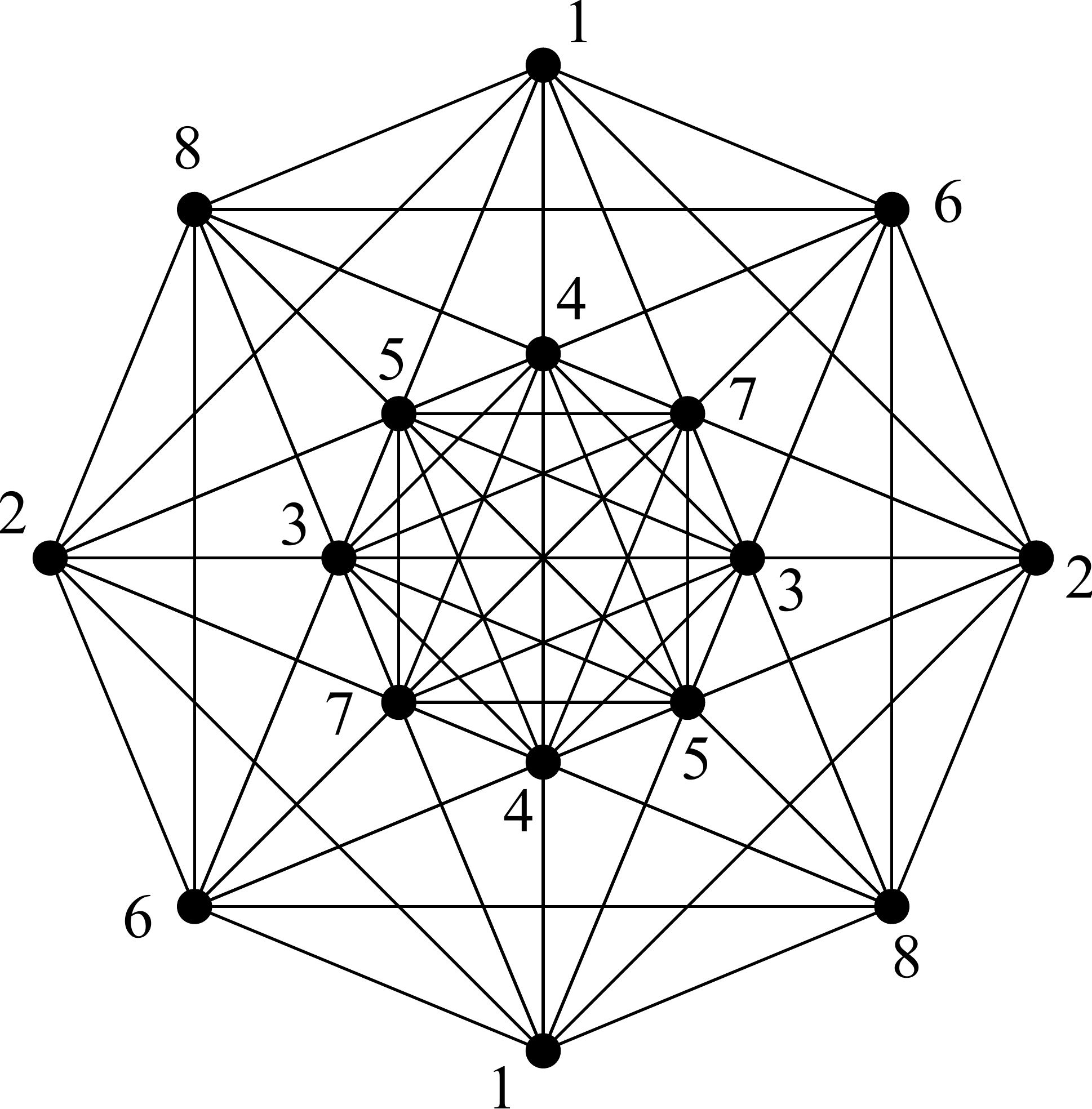}
 \end{center}
 \caption{The chosen colouring for $P^5$.}
 \label{P5_colouring:fig}
\end{figure}

A convenient planar projection of its 1-skeleton is shown in Figure \ref{P5:fig}. We assign to $1_{21}$, and hence to $P^5$, the 8-colouring depicted in Figure \ref{P5_colouring:fig}. This produces a hyperbolic 5-manifold $M^5$ tessellated into $2^8 = 256$ copies of $P^5$.

\begin{figure}
 \begin{center}
  \includegraphics[width = 12.5 cm]{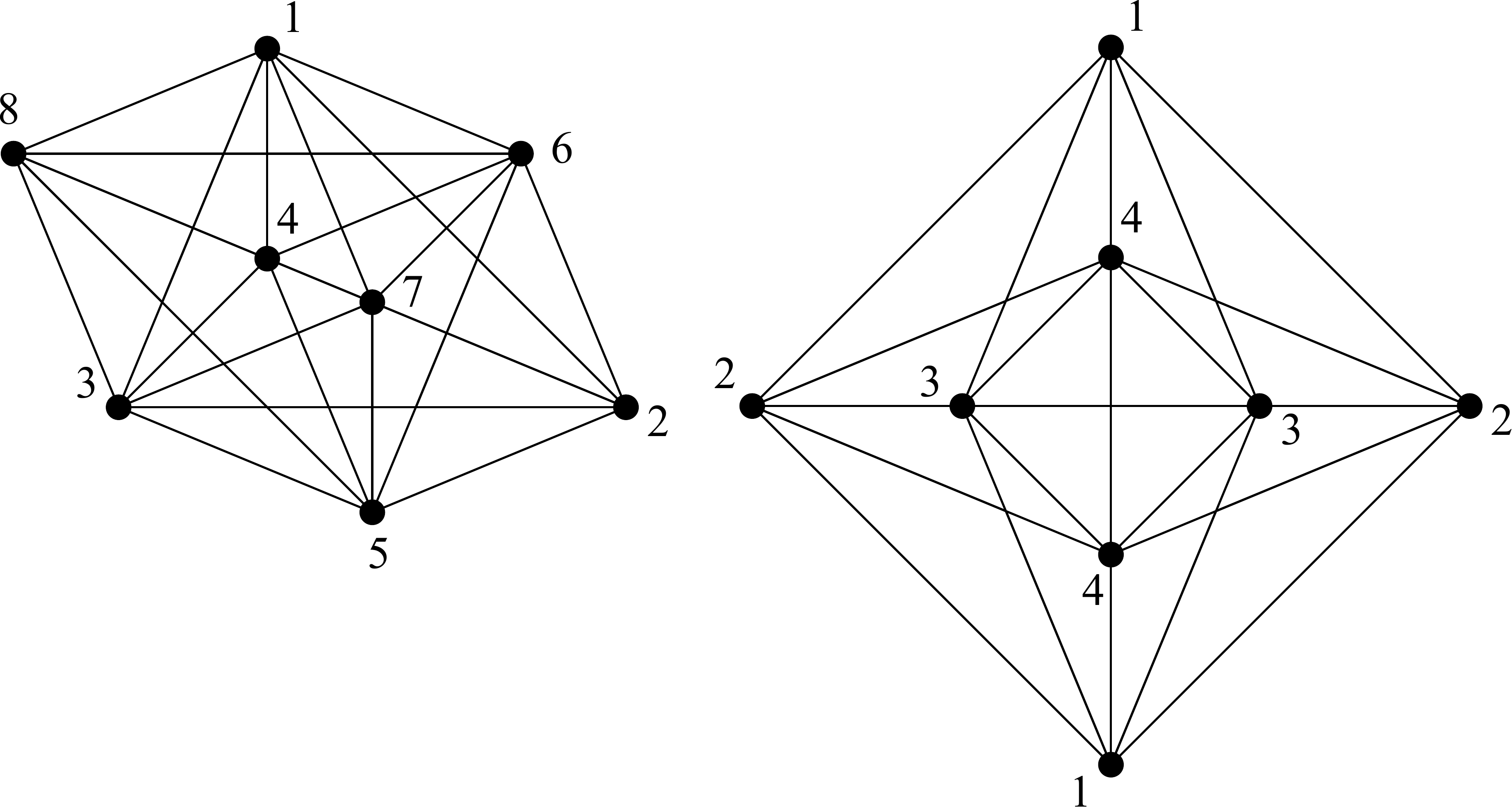}
 \end{center}
 \caption{The ten 4-octahedral facets of $1_{21}$ are of two types. Eight are obtained by rotating the type shown on the left, and two by rotating the type shown on the right. These are subgraphs of the 1-skeleton in Figure \ref{P5_colouring:fig}. Some edges are superposed.}
 \label{P5_facets:fig}
\end{figure}

The polytope $P^5$ has 10 ideal vertices. Each ideal vertex is dual to a 4-octahedral facet of $1_{21}$ contained in a hyperplane $x_i = \pm 1$. We deduce then that there are two types of 4-octahedral facets, depicted in Figure \ref{P5_facets:fig}. Eight facets are of the left type, and two of the right type (all obtained by rotating the graphs shown in the figure). The vertices of the facets of the first type inherit a 8-colouring, while those of the facets of the second type inherit a 4-colouring. 

We have discovered that there are 8 ideal vertices of the first type and 2 ideal vertices of the second type in $P^5$. The link of an ideal vertex of the first type of $P^5$ is a 8-coloured 4-cube $C$, while the link of an ideal vertex of the second type is a 4-coloured 4-cube $C$. Note that 4 and 8 are precisely the minimum and maximum number of colours on a 4-cube. By Corollary \ref{cusps:cor}, the counterimage of $C$ consists of a single (since $2^{8-8}=1$) toric cusp section in the first case, and of $2^{8-4} = 2^4 = 16$ toric cusp sections in the second case. Therefore the hyperbolic manifold $M^4$ has $8\cdot 1 + 2\cdot 16  = 40$ cusps overall. The first $8$ cusps lie above the 8 vertices of the first type, and the remaining 32 cusps lie above the 2 vertices of the second type, distributed as 16 above each.

Using Sage we have calculated the Betti numbers of $M^5$:
$$b_0 = 1, \qquad b_1 = 24, \qquad b_2 = 120, \qquad b_3 = 136, \qquad b_4 = 39.$$

We get of course $\chi(M^5)=0$.

\subsection{The manifold $M^6$} \label{M6:subsection}
The hyperbolic polytope $P^6\subset \matH^6$ is fully described in \cite{ERT, PV}, and we refer to these sources for more details. It has 27 facets, each isometric to $P^5$. It also has 72 finite vertices and 27 ideal vertices. Every ideal vertex is opposed to a facet.

The dual Gosset polytope $2_{21}$ has 27 vertices. We can represent them in the affine hyperspace of $\matR^7$ of equation $x_1 + \cdots + x_6 - 3x_7 = -1$, as the vertices 
$$(-1,0,0,0,0,0,0), \quad (1,1,0,0,0,0,1), \quad (0,1,1,1,1,1,2)$$ 
and all the other vertices obtained from these by permuting the first 6 coordinates, so we get $6+15+6=27$ vertices in total, see \cite[Table 2]{ERT}.
Two vertices are connected by an edge if their Lorentzian product in $\matR^7$ with signature $(++++++-)$ is zero. The Gosset polytope $2_{21}$ has 99 facets; of these, 72 are regular 5-simplexes dual to the finite vertices of $P^6$, and 27 are regular 5-octahedra dual to the ideal vertices of $P^6$.

\begin{figure}
 \begin{center}
  \includegraphics[width = 10 cm]{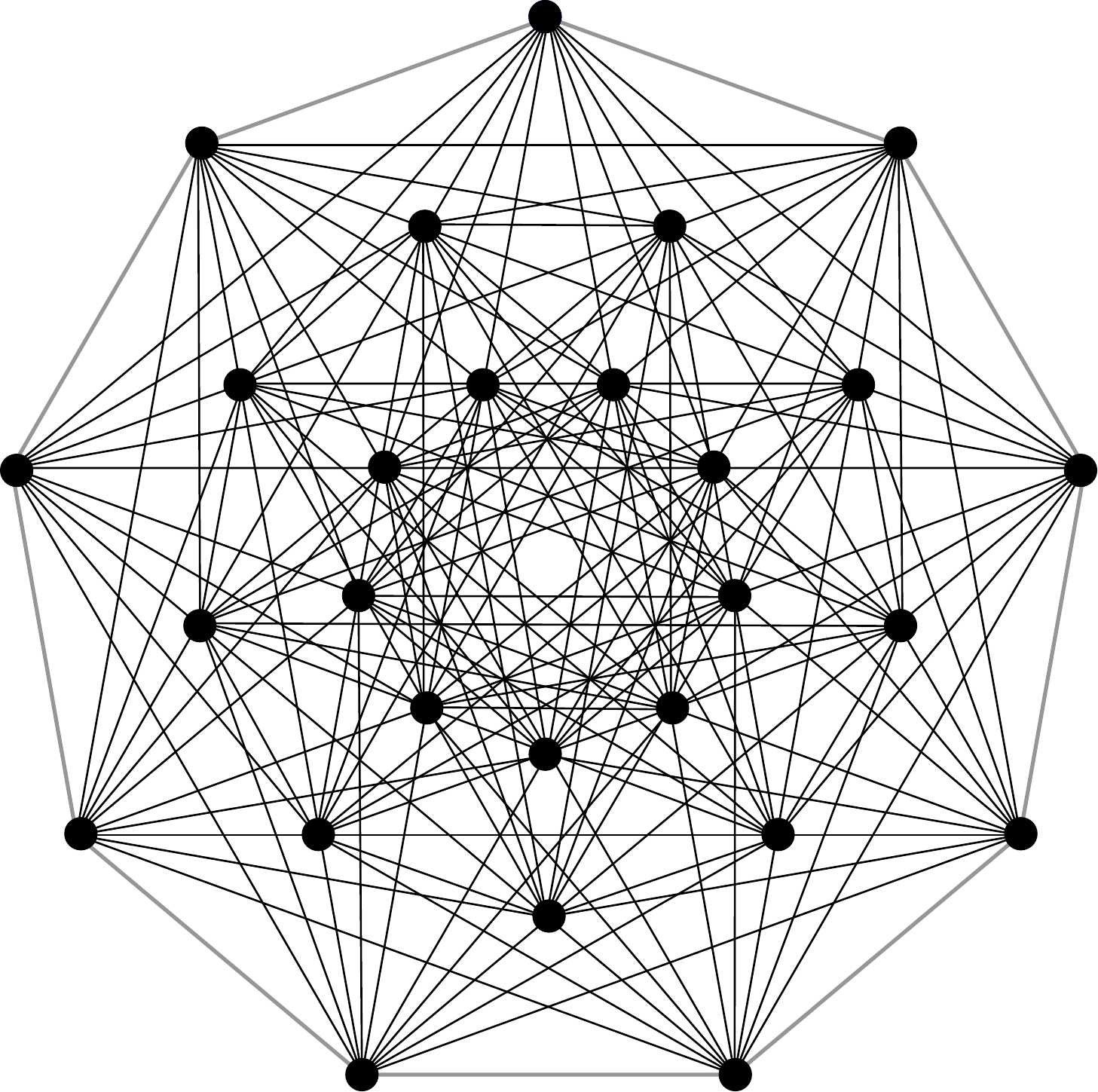}
 \end{center}
 \caption{An orthogonal projection of the 1-skeleton of $2_{21}$ on the plane. Some edges are superposed. There are 9 lines intersecting in the centre of the figure, each line containing three vertices that are mutually non incident.}
 \label{P6:fig}
\end{figure}

Both $P^6$ and $2_{21}$ have many remarkable properties. To mention one, the 1-skeleton of $2_{21}$ is the configuration graph of the 27 lines in a general cubic surface, see \cite{C}. A planar projection of the 1-skeleton of $2_{21}$ taken from \cite{C} is shown in Figure \ref{P6:fig}. In the figure we see that there are 9 lines that intersect in the centre, containing each 3 mutually non-adjacent vertices. This suggests that the polytope may have a nice 9-colouring. 

Inspired by the figure, we describe a 9-colouring for $2_{21}$. The three vertices
$$(-1,0,0,0,0,0,0), \quad (1,1,0,0,0,0,1), \quad (1,0,1,1,1,1,2)$$
are mutually non connected by any edge since their Lorentzian products are not zero. We assign them the colour 1. If we permute cyclically the first 6 entries of these three vertices we get 5 more triplets of mutually non connected vertices, and we assign them the colours $2,\ldots, 6$. Finally, we assign the colours $7,8,9$ to the following remaining triplets of mutually disjoint vertices:
\begin{gather*}
(1,0,1,0,0,0,1), \quad (0,1,0,0,1,0,1), \quad (0,0,0,1,0,1,1); \\
(0,1,0,1,0,0,1), \quad (0,0,1,0,0,1,1), \quad (1,0,0,0,1,0,1); \\
(0,0,1,0,1,0,1), \quad (1,0,0,1,0,0,1), \quad (0,1,0,0,0,1,1).
\end{gather*}

We equip $P^6$ with this 9-colouring. Each triple of facets with the same colour is called a \emph{triplet}. The colouring produces a hyperbolic 6-manifold $M^6$, tessellated into $2^9 = 512$ copies of $P^6$. We have $\chi(M^6) = -512/8 = -64$.

The polytope $P^6$ has 27 ideal vertices. Using our program in Sage \cite{code} we discover that the link of each of the 27 ideal vertices of $P^6$ is a 9-coloured 5-cube $C$. We show one explicit example. Every facet $F$ of $P^6$ is opposite to an ideal vertex, which is incident precisely to those facets that are not incident to $F$. Correspondingly, every vertex $v$ in $2_{21}$ is opposite to a 5-octahedral facet, whose vertices are precisely those that are not connected to $v$. The 5-octahedral facet opposite to the vertex $v=(-1,0,0,0,0,0,0)$ has the following vertices: 
\begin{gather*}
(1,1,0,0,0,0,1),\! (1,0,1,0,0,0,1),\! (1,0,0,1,0,0,1),\! (1,0,0,0,1,0,1),\! (1,0,0,0,0,1,1), \\
(1,0,1,1,1,1,2),\! (1,1,0,1,1,1,2),\! (1,1,1,0,1,1,2),\! (1,1,1,1,0,1,2),\! (1,1,1,1,1,0,2).
\end{gather*}
The two vertices that lie in the same column are not connected. Their colours are
\begin{gather*}
1,\ 7,\ 9,\ 8,\ 6, \\
1,\ 2,\ 3,\ 4,\ 5.
\end{gather*}
All the 9 colours are present. As we said above, using our Sage program we discover that a similar configuration holds at every vertex $v$. Therefore by Corollary \ref{cusps:cor} the counterimage of $C$ consists of a single (since $2^{9-9}=1$) toric cusp section. We deduce finally that the hyperbolic manifold $M^6$ has 27 cusps, one above each vertex of $P^6$. The Betti numbers of $M^6$, calculated by our program, are:
$$b_0 = 1, \qquad b_1 = 18, \qquad b_2 = 183, \qquad b_3 = 411, \qquad b_4 = 207, \qquad b_5 = 26.$$

We get $\chi(M^6)=-64$ again.

\subsection{The manifold $M^7$}
The hyperbolic polytope $P^7 \subset \matH^7$ is described in \cite{ERT, PV}. It
has 56 facets, each isometric to $P^6$. It also has 576 finite vertices and 126 ideal vertices. 

The dual Gosset polytope $3_{21}$ has 56 vertices. We will discover below that the 56 vertices can be partitioned into 14 sets of 4 mutually disjoint vertices, called \emph{quartets}. This partition will be induced from a colouring of $P^8$, that is in turn easily described using octonions. The precise description of the partition is given below in Section \ref{M7:again:subsection}.

We equip $P^7$ with the 14-colouring induced by the partition into 14 quartets. This produces a hyperbolic 7-manifold $M^7$, tessellated into $2^{14}=16384$ copies of $P^7$.

The polytope $P^7$ has 126 ideal vertices. Using our Sage program we discover that, similarly as with $P^5$, there are two types of ideal vertices with respect to the chosen 14-colouring of $P^7$. The first type consists of 112 vertices, and the second type of only 14. The link of an ideal vertex of the first type is a 12-coloured 6-cube $C$, while the link of an ideal vertex of the second type is a 6-coloured 6-cube. Note that, as with $P_5$, the numbers 6 and 12 are the minimum and maximum number of colours in a 6-cube. From Corollary \ref{cusps:cor} we deduce that $M^7$ has $14 \cdot 2^{14-6} + 112 \cdot 2^{14-12} = 4032$ cusps overall.

Using Sage we have calculated the Betti numbers of $M^7$:
$$b_0 = 1, \ b_1 = 182, \ b_2 = 6321, \ b_3 = 41300, 
\ b_4 = 55139, \ b_5 = 24010, \ b_6 = 4031.$$

We get of course $\chi(M^7)=0$.

\subsection{The manifold $M^8$}
The hyperbolic polytope $P^8\subset \matH^8$ is described in \cite{ERT, PV}. It has 240 facets, each isometric to $P^7$. It also has 17280 finite vertices and 2160 ideal vertices. 

The dual Gosset polytope $4_{21}$ has 240 vertices. This beautiful albeit complicated polytope can be described elegantly using octonions, much in the same way as the 4-dimensional 24-cell may be defined using quaternions. This viewpoint is crucial in this paper, so we introduce it carefully. 

\subsubsection*{A 3-colouring for the 24-cell} To warm up, we start by recalling that the 24 vertices of the 24-cell are the quaternions 
$$\pm 1, \ \pm i,\ \pm j,\ \pm k,\ \tfrac 12 (\pm 1 \pm i \pm j \pm k).$$

Two such vertices are adjacent along an edge if and only if their Euclidean scalar product is $\frac 12$ (we identify the quaternions space with the Euclidean $\matR^4$ as usual). Every vertex is adjacent to 8 other vertices. 

We can assign 3 colours to the 24 vertices, by subdividing them into 3 sets of 8 vertices each, that we call \emph{octets}.
These are:
\begin{enumerate}
\item $ \pm 1, \pm i, \pm j, \pm k$;
\item the elements $\tfrac 12 (\pm 1 \pm i \pm j \pm k) $ with an even number of minus signs;
\item the elements $\tfrac 12 (\pm 1 \pm i \pm j \pm k) $ with an odd number of minus signs.
\end{enumerate}

The scalar product of two vertices lying in the same octet is an integer, so it is never $\frac 12$. Therefore this indeed defines a 3-colouring of the vertices of the 24-cell. Since the dual of a 24-cell is another 24-cell, we also get a 3-colouring of the facets of the dual 24-cell. This colouring was heavily employed in \cite{KM}.

Here is an algebraic description of this 3-colouring that will be useful below. The 24 vertices of the 24-cell described above form a group called the \emph{binary tetrahedral group}. The 8 elements $ \pm 1, \pm i, \pm j, \pm k$ form a normal subgroup of index 3, called the \emph{quaternion group} and indicated with the symbol $Q_8$. The octets are just the three lateral classes of $Q_8$.

\begin{figure}
 \begin{center}
  \includegraphics[width = 6 cm]{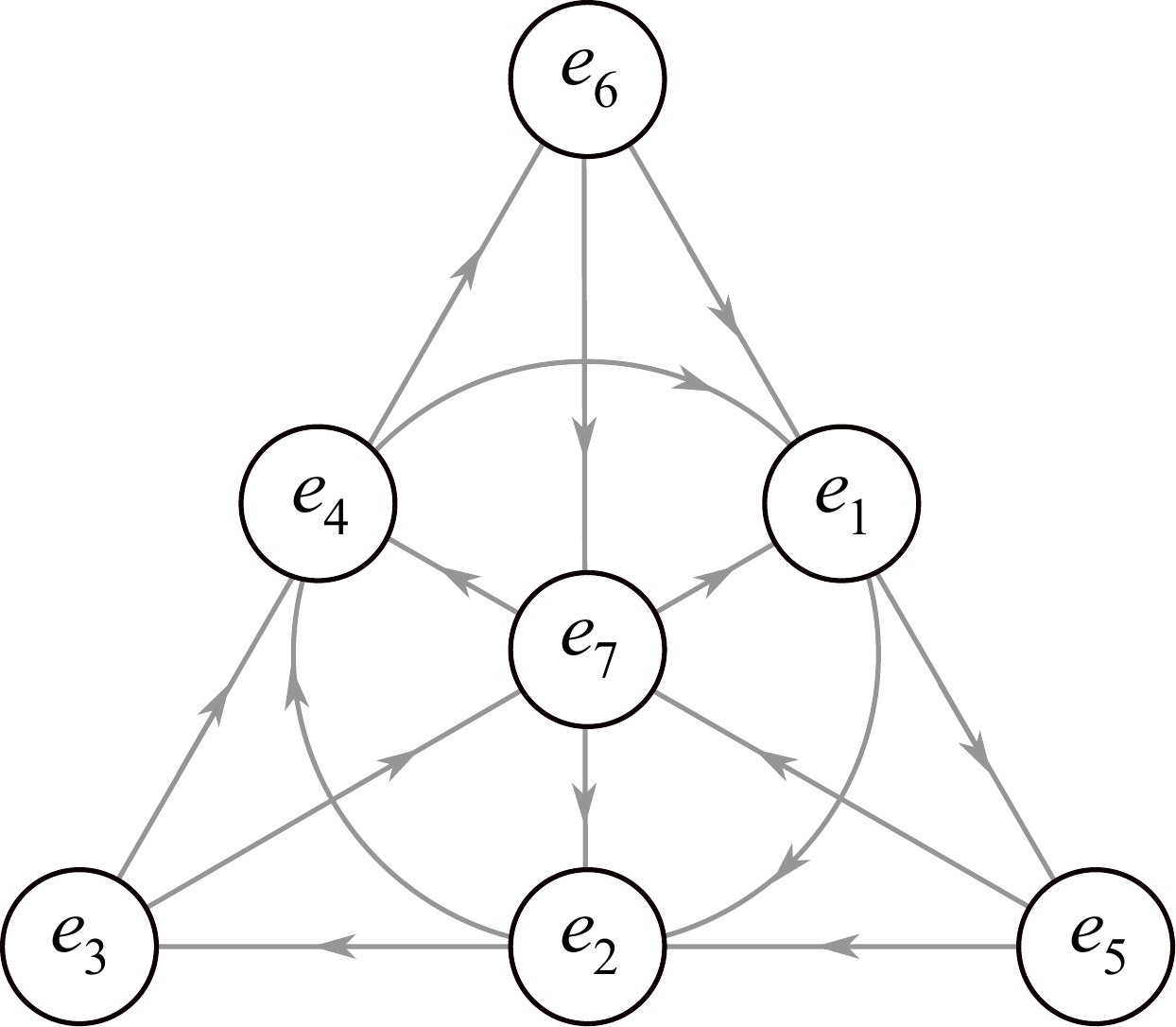}
 \end{center}
 \caption{The Fano plane. The circle should be interpreted as a line.}
 \label{Fano:fig}
\end{figure}

\subsubsection*{Octonions}
We now turn to the Gosset polytope $4_{21}$ and the octonions. For a nice introduction to the subject we recommend \cite{B}.
We describe an octonion as a linear combination of $1, e_1, e_2, \ldots, e_7$. We have $e_i^2 = -1$, and the multiplication of two distinct elements $e_i$ and $e_j$ is beautifully described by the Fano plane shown in Figure \ref{Fano:fig}. The Fano plane is the projective plane over $\matZ_2$ and it contains 7 points and 7 oriented lines: every line is a cyclically ordered triple of points as in the figure. For every $i\neq j$ we have $e_ie_j = \pm e_k$, where $e_k$ is the third vertex in the unique line containing $e_i$ and $e_j$, and the sign is positive if and only if the line is cyclically oriented like $e_i\to e_j \to e_k$. So for instance $e_1e_2 = e_4$ and $e_1e_6 = -e_5$. In general, we get
$$e_{n} e_{n+1} = e_{n+3}$$
where the subscripts run modulo 7. The product is neither commutative nor associative: for every $i,j,k$ we have
$$(e_ie_j)e_k = \pm e_i(e_je_k)$$
where the sign is $+1$ if and only if $e_i,e_j,e_k$ belong to the same line in the Fano plane (which is always the case if $i,j,k$ are not distinct).

\subsubsection*{A 15-colouring for the Gosset polytope $4_{21}$}
The 240 vertices of the Gosset polytope $4_{21}$ are the octonions
\begin{gather*}
 \pm 1,\ \pm e_1,\ \pm e_2, \ \pm e_3, \ \pm e_4, \ \pm e_5, \ \pm e_6, \ \pm e_7, \\
\tfrac 12(\pm 1 \pm e_{n} \pm e_{n+1} \pm e_{n+3}), \quad
\tfrac 12(\pm e_{n+2} \pm e_{n+4} \pm e_{n+5} \pm e_{n+6})
\end{gather*}
where $n$ runs modulo 7. Although we will not need this information, we mention that these are (up to rescaling) precisely the 240 non-trivial elements of smallest norm in the $E_8$ lattice.

We have 16 elements of type $\pm 1$ or $\pm e_i$. Each line $l$ in the Fano plane contains three vertices $e_n, e_{n+1}, e_{n+3}$ and determines 16 elements of type $\tfrac 12(\pm 1 \pm e_{n} \pm e_{n+1} \pm e_{n+3})$ and 16 elements of type $\tfrac 12(\pm e_{n+2} \pm e_{n+4} \pm e_{n+5} \pm e_{n+6})$, so we indeed get $16 + 7 \cdot 16 + 7 \cdot 16 = 15 \cdot 16 = 240$ vertices overall.

Two vertices of $4_{21}$ are connected by an edge if and only if their Euclidean scalar product is $\frac 12$. One can check that every vertex is adjacent to 56 other vertices (its link is dual to $P^7$ that has 56 facets).

Similarly to what we did with the 24-cell, we can assign a 15-colouring to $4_{21}$ by subdividing the 240 vertices into 15 sets of 16 elements each; we call each such set a \emph{hextet}. The hextets are:
\begin{enumerate}
\item $\pm 1,\pm e_1, \pm e_2, \pm e_3, \pm e_4, \pm e_5, \pm e_6, \pm e_7$;
\item the elements $\tfrac 12(\pm 1 \pm e_{n} \pm e_{n+1} \pm e_{n+3})$ and
$\tfrac 12(\pm e_{n+2} \pm e_{n+4} \pm e_{n+5} \pm e_{n+6})$ with an even number of minus signs;
\item the elements $\tfrac 12(\pm 1 \pm e_{n} \pm e_{n+1} \pm e_{n+3})$ and
$\tfrac 12(\pm e_{n+2} \pm e_{n+4} \pm e_{n+5} \pm e_{n+6})$ with an odd number of minus signs.
\end{enumerate}
The hextets of type (2) and (3) depend on the choice of $n$ modulo 7. So we get $1+2\cdot 7 = 15$ hextets overall. One can check that the scalar product of two vertices lying in the same hextet is always an integer, so it is never $\frac 12$. Therefore we can assign the same colour to all the 16 members of a given hextet, and hence obtain a 15-colouring for $4_{21}$ as promised.

\subsubsection*{Algebraic description}
There is an algebraic interpretation for the 15-colouring of $4_{21}$ analogous to that for the 3-colouring of the 24-cell. We warn the reader that some caution is needed when passing from quaternions to octonions: first, the product of octonions is notoriously nonassociative; second, contrary to a common mistake (see \cite[Chapter 9]{CS} for a discussion), and as proved by Coxeter \cite{C2}, the 240 vertices of $4_{21}$ are \emph{not} closed under multiplication! Indeed the product of the two vertices
$$\frac 12 (1+e_1+e_3+e_7) \cdot \frac 12(1+e_1+e_2+e_4) = \frac 12(e_1 + e_3+ e_4 + e_6)$$
is not a vertex. We could fix this via a single reflection that transforms the 240 vertices into a multiplicatively closed set (this is explained in \cite[Section 9.2]{CS}), thus obtaining another isometric description of $4_{21}$, but we do not really need this here, so we just keep them as they are. 

The only thing that we need here is that the 240 octonions are closed under left multiplication by each of the 16 elements in the hextet $S=\{\pm 1, \pm e_i\}$, a fact that can be verified easily. The set $S$ is closed under multiplication, but it is not a group since it is not associative.
One can also verify that the left multiplication by each element of $S$ preserves each hextet, and that this ``action'' of $S$ is free and transitive, in the sense that for very pair of distinct elements in a hextet there is a unique element of $S$ that sends the first to the second by left-multiplication.

Summing up, the 15 hextets that we have constructed are the orbits of the action of $S$ by left-multiplication on the set of 240 vertices of $4_{21}$. This is analogous to the 3-colouring of the 24-cell, where the 3 octets are the orbits of the action of the quaternion group $Q_8$ by left multiplication.

\subsubsection*{The manifold $M^8$}
We equip $P^8$ with the 15-colouring just defined. This produces a hyperbolic 8-manifold $M^8$, tessellated into $2^{15} = 32768$ copies of $P^8$. We have $\chi(M^8) = 2^{15}\cdot 17/2 = 278528$. 

The polytope $P^8$ has 2160 ideal vertices. Using our Sage program we discover a phenomenon that was already present with $P^5$ and $P^7$. The ideal vertices are of two types: the first type contains 1920 of them and the second type 240. The link of a vertex of the first type is a 14-coloured 7-cube, while the link of a vertex of the second type is a 7-coloured 7-cube. As with $P^5$ and $P^7$, we note that 7 and 14 are the minimum and maximum possible number of colours in a 7-cube. From Corollary \ref{cusps:cor} we deduce that $M^8$ has $240\cdot 2^{15-7} + 1920 \cdot 2^{15-14} = 65280$ cusps.

Using Sage we have calculated the Betti numbers of $M^8$:
\begin{gather*}
b_0 = 1,\quad b_1 = 365, \quad b_2 =  33670, \quad b_3 = 583290, \\
b_4 = 1783226, \quad b_5 = 1346030, \quad b_6 = 456595, \quad b_7 = 65279.
\end{gather*}

We get $\chi(M^8)=278528$ again.

\subsection{Back to the polytope $P^7$} \label{M7:again:subsection}
The polytope $P^7$ is a facet of $P^8$. We think of $P^7$ as the facet dual to the vertex 1 of $4_{21}$. As we already said, we equip $P^7$ with the colouring induced by the 15-colouring of $P^8$ just introduced.

We study this inherited colouring of $P^7$. We think of $3_{21}$ as the link figure of the vertex $1$ of $4_{21}$. The vertices of $4_{21}$ adjacent to $1$ are precisely those of the form
$$\tfrac 12(1 \pm e_{n} \pm e_{n+1} \pm e_{n+3})$$
where $n$ runs modulo 7. So we get $7\cdot 8 = 56$ vertices, as required. These vertices are contained in the hyperplane $x_0 = \frac 12$ and their convex hull is $3_{21}$. Two such vertices are connected by an edge in $3_{21}$ if and only if their scalar product is $\frac 12$.

The 15-colouring of $4_{21}$ induces a 14-colouring of $3_{21}$ that partitions the 56 vertices into 14 sets of four vertices each, that we call \emph{quartets}. Each quartet consists of the vertices $\frac 12(1 \pm e_n \pm e_{n+1} \pm e_{n+3})$ that share the same $n$ and the same parity of the minus signs.

\subsection{Volumes}
We have constructed a colouring on each polytope $P^3, \ldots, P^8$, and hence obtained a list of manifolds $M^3,\ldots, M^8$. Table \ref{colourings:table} summarizes the colouring type of each polytope. 

\begin{table}
\begin{center}
\begin{tabular}{c|c|c|c|c|c}
$P^3$ & $P^4$ & $P^5$ & $P^6$ & $P^7$ & $P^8$ \\
\hline \hline
3 pairs & 5 pairs & 8 pairs & 9 triplets & 14 quartets & 15 hextets
\end{tabular}
\vspace{.5 cm}
\caption{The colouring type of each $P^3,\ldots, P^8$.}
\label{colourings:table}
\end{center}
\end{table}

The volumes of the hyperbolic manifolds $M^3, \ldots, M^8$ are listed in Table \ref{Mn:table}. In even dimension $n=2m$ we have used the Gauss-Bonnet formula
$$\Vol(P) = (-2\pi)^m / (n-1)!! \cdot \chi(P).$$
In odd dimension, we have
$$\Vol(P^3) = L(2)\sim 0.91, \quad \Vol(P^5) = 7 \zeta(3) / 8 \sim 1.05, \quad
\Vol(P^7) = 8L(4) \sim 7.92.$$
The symbols $\zeta$ and $L$ indicate the Riemann and Dirichlet functions, see \cite{RT, ERT}.

\begin{table}
\begin{center}
\begin{tabular}{c||cccc}
        & Volume & $\chi$ & Cusps \\
 \hline \hline
$M^3$ & $8L(2) \sim 7.28$ & 0 & 3 \\
$M^4$ & $8\pi^2/3 \sim 26.3$ & $2$ & 5 \\
$M^5$ & $224\zeta(3) \sim 269$ & 0 & 40  \\
$M^6$ & $512\pi^3/ 15\sim 1.06 \cdot 10^3$ & $-64$ & 27 \\
$M^7$ & $131072L(4) \sim 1.30 \cdot 10^5$ & 0 & 4032 \\
$M^8$ & $4456448\pi^4\!/105 \sim 4.13\cdot 10^6$  & $278528$ & 65280
\end{tabular}
\vspace{.5 cm}
\caption{The volume, the Euler characteristic, and the number of cusps of each hyperbolic $n$-manifold $M^n$.}
\label{Mn:table}
\end{center}
\end{table}

\subsection{The chosen colourings are all minimal}
Although we will not need it, we mention the following fact.

\begin{prop}
The colourings for $P^3,\ldots, P^8$ defined in the previous sections have the smallest possible number of colours for each polytope.
\end{prop}
\begin{proof}
We can verify by hand when $n\leq 5$ and using our Sage program when $6\leq n \leq 8$ that the maximum number of pairwise disjoint facets in $P^n$ is equal to 2, 2, 2, 3, 4, 16 when $n=3,4,5,6,7,8$. These are precisely the cardinalities of the facets sharing the same colour for all $n$, see Table \ref{colourings:table}. Therefore we cannot find a more efficient colouring than the one listed in the table. 
\end{proof}

\subsection{The last non-zero Betti number}
The Betti numbers $b_i$ and the number $c$ of cusps of each $M^n$ were listed in Table \ref{Mn:intro:table}. In all the cases we have $b_{n-1} = c-1$. Since $M^n$ is the interior of a compact manifold with $c$ boundary components, in general we must have $b_{n-1} \geq c-1$. Therefore here the Betti number $b_{n-1}$ is as small as possible, given the number $c$ of cusps.

\section{The algebraic fibrations} \label{f:section}
We have constructed some hyperbolic manifolds $M^3,\ldots, M^8$, and our aim is now to build some nice maps $f\colon M^n \to S^1$ for all $n=3,\ldots, 8$. We produce these maps by assigning to each $P^n$ an appropriate \emph{state}, as prescribed by  \cite{JNW}. We then study the maps by applying some fundamental results of \cite{BB}. 

\subsection{States} \label{states:subsection}
Let $P \subset \matX^n$ be a right-angled polytope in some space $\matX^n = \matH^n, \matR^n$, or $S^n$. 
Following \cite{JNW}, a \emph{state} is a partition of the facets of $P$ into two subsets, that we denote as I (in) and O (out). Every facet thus inherits a \emph{status} I or O.

Let $P$ be equipped with a colouring with $c$ colours. This induces a free action of $\matZ_2^c$ on the set of all the states of $P$, in the following way. For every $j\in \{1,\ldots, c\}$, the basis element $e_j$ acts by reversing the I/O status of every facet of $P$ coloured by $j$, while leaving the status of the other facets unaffected. The action is free, hence each orbit consists of $2^c$ distinct states.

\subsection{Diagonal maps} \label{diagonal:subsection}
As discovered in \cite{JNW}, the choice of a colouring and a state for a right-angled polytope $P$ induce both a manifold $M$ and a \emph{diagonal map} $M \to S^1$. 
(The construction of \cite{JNW} is actually more general than this, but this interpretation is enough here.)
Shortly:

\begin{tcolorbox}
\begin{center}
colouring + state on $P$ 
$\Longrightarrow$ manifold $M$ + diagonal map $f\colon M \to S^1$.
\end{center}
\end{tcolorbox}

We explain how this works. We already know how a colouring on $P$ produces a manifold $M$, so it remains to explain how a state induces a map $f\colon M\to S^1$. 

The manifold $M$ is tessellated into the $2^c$ polytopes $P_v$ with varying $v\in \matZ_2^c$. Since these are right-angled, the tessellation is dual to a cube complex $C$ with $2^c$ vertices. We work in the piecewise-linear category (see \cite{RS} for an introduction) and think of $C$ as piecewise-linearly embedded inside $M$. If $P$ has some ideal vertices (as it will be the case with all the polytopes $P^n\subset\matH^n$ that we consider here), the complement $M \setminus C$ consists of open cusps, so there is a deformation retraction $r\colon M \to C$. The cube complex $C$ is a spine of $M$.

We indicate the vertex of $C$ dual to $P_v$ simply as $v$, so 
the vertices of $C$ are identified with $\matZ_2^c$. Here $v$ stands both for a $v$ector of $\matZ_2^c$ and a $v$ertex of $C$.

\begin{figure}
 \begin{center}
  \includegraphics[width = 5.5 cm]{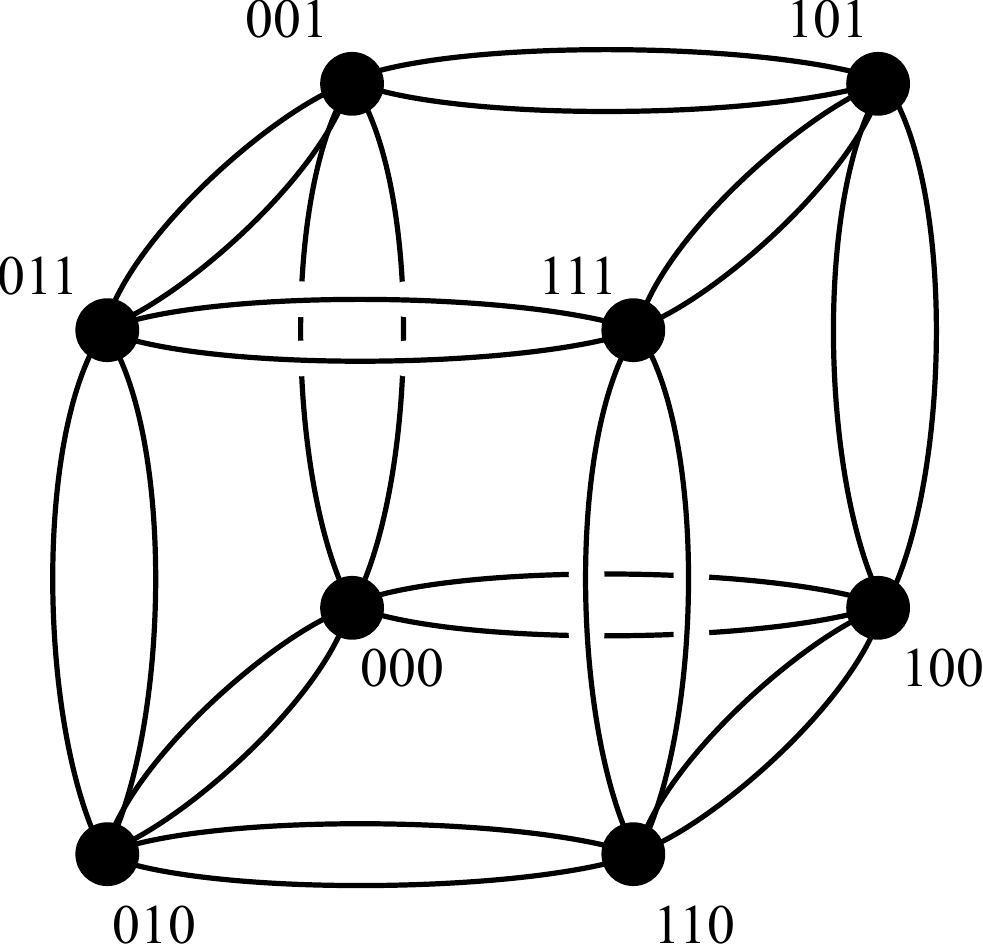}
 \end{center}
 \caption{The 1-skeleton of the dual cubulation $C$ for $M^3$, tessellated into 8 polyhedra $P_v^3$. The polyhedron $P^3$ is 3-coloured, and each colour is painted on two faces. The vertices of $C$ are identified with $\matZ_2^3$. There are two edges connecting $v$ and $v+e_j$ corresponding to the two faces in $P_v$ with the same colour $j$, for each $j=1,2,3$. }
 \label{C:fig}
\end{figure}

The edges of $C$ are dual to the facets of the tessellation: an edge of $C$ connects $v$ and $v+e_j$ if the dual facet $F$ is coloured as $j$. So in particular there are $k$ distinct edges connecting $v$ to $v+e_j$, where $k$ is the number of facets in $P$ coloured with $j$. In all the colourings that we have chosen for the polytopes $P^n$ the number $k$ does not depend on the colour $j$. The 1-skeleton of $C$ for $P=P^3$ is shown in Figure \ref{C:fig}.

\begin{example} If we consider $P = P^8$ with its 15-colouring, there are $2^{15}$ vertices in $C$, and 16 edges connecting $v$ to $v+e_j$ for every $v$ and every $j$.
\end{example}

Let now $s$ be a fixed state for $P$. The state $s$ induces an orientation on all the edges of $C$, in the simplest possible way: consider an edge connecting $v$ and $v+e_j$, where the $j$-th component of $v$ is zero, that is $v_j=0$. The edge is dual to some facet of the tessellation that is a precise identical copy of a facet $F$ of $P$. If the status of $F$ is O, we orient the edge \emph{outward}, that is from $v$ to $v+e_j$, while if it is I we orient it \emph{inward}, from $v+e_j$ to $v$.

\begin{figure}
 \begin{center}
  \includegraphics[width = 7 cm]{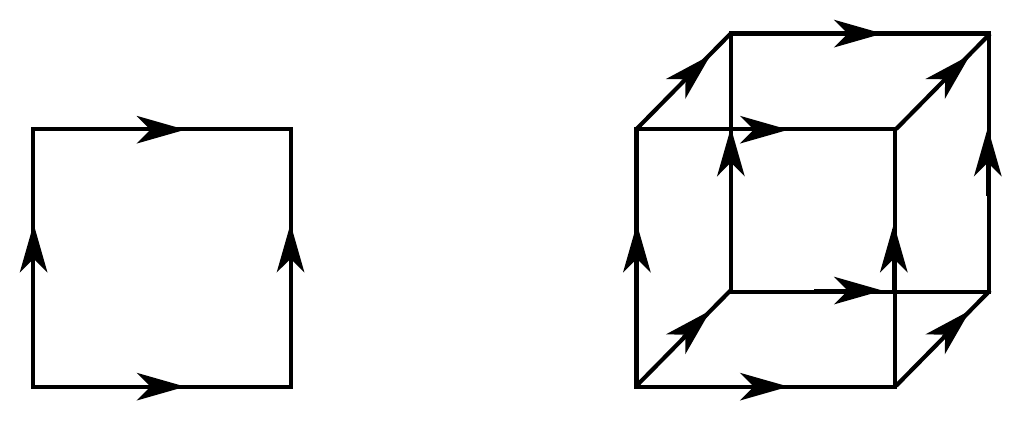}
 \end{center}
 \caption{Every square (and hence every $k$-cube) of the cubulation has its opposite edges oriented coherently as shown here. }
 \label{arrows:fig}
\end{figure}

By construction, this orientation is \emph{coherent}, that is on every square of $C$ (and hence on any $k$-cube) the orientations of two opposite edges match as in Figure \ref{arrows:fig}. This crucial fact allows us to apply Bestvina -- Brady theory \cite{BB}. We identify every $k$-cube of $C$ with the standard $k$-cube $[0,1]^k \subset \matR^k$, so that the orientations on the edges of $C$ match with the orientations of the axis in $\matR^k$. The diagonal map on the standard $k$-cube is
$$[0,1]^k \longrightarrow S^1 = \matR/{\matZ}, \qquad x \longmapsto x_1+ \cdots + x_k.$$
The diagonal maps on the $k$-cubes of $C$ match to give a well-defined continuous piecewise-linear map $C \to S^1$. By pre-composing it with the deformation retraction $r\colon M \to C$, we finally get a \emph{diagonal map} 
$$f\colon M \to S^1.$$
This is the main protagonist of our construction. The diagonal map induces a homomorphism $f_* \colon \pi_1(M) \to \pi_1(S^1) = \matZ$. A dichotomy arises here:

\begin{prop} \label{dichotomy:prop}
Precisely one of the following holds:
\begin{enumerate}
\item The facets of $P$ with the same colour also have the same status. In this case $f$ is homotopic to a constant.
\item There are at least two facets in $P$ with equal colour and opposite status. In this case the homomorphism $f_*\colon \pi_1(M) \to \pi_1(S^1) = \matZ$ is non-trivial with image $2\matZ$.
\end{enumerate}
\end{prop}
\begin{proof}
If (1) holds, all the edges joining two given vertices of $C$ are oriented in the same way, and we may lift the map $f\colon M \to S^1$ to a map $\tilde f \colon M \to \matR$ as follows: send every vertex $v\in \matZ_2^c$ of $C$ to the maximum number of edges entering in $v$ and pointing inward from distinct vertices, then extend $\tilde f$ diagonally to cubes. Since $f$ can be lifted, it is homotopic to a constant.

If (2) holds, there are two edges joining the same pair of vertices with opposite orientation, that form a loop that is sent to $\pm 2$ along $f_*$. Moreover $1 \not \in {\rm Im}(f_*)$ because the 1-skeleton of $C$ is naturally bipartited into even and odd vertices, according to the parity of $v_1+\cdots + v_c$.
\end{proof}

The case (1) is not so interesting: all the examples that we construct here on the hyperbolic manifolds $M^n$ will be of type (2). In (2), since ${\rm Im} f_* = 2 \matZ$, one may decide to replace $f$ with a lift along a degree-2 covering of $S^1$ to get a surjective $f_*$.

\begin{cor} \label{trivial:cor}
If all the facets of $P$ have distinct colours, the diagonal map $f$ is always homotopically trivial, for every choice of a state.
\end{cor}

This inefficient colouring is therefore of no use here.

\begin{figure}
 \begin{center}
  \includegraphics[width = 5.5 cm]{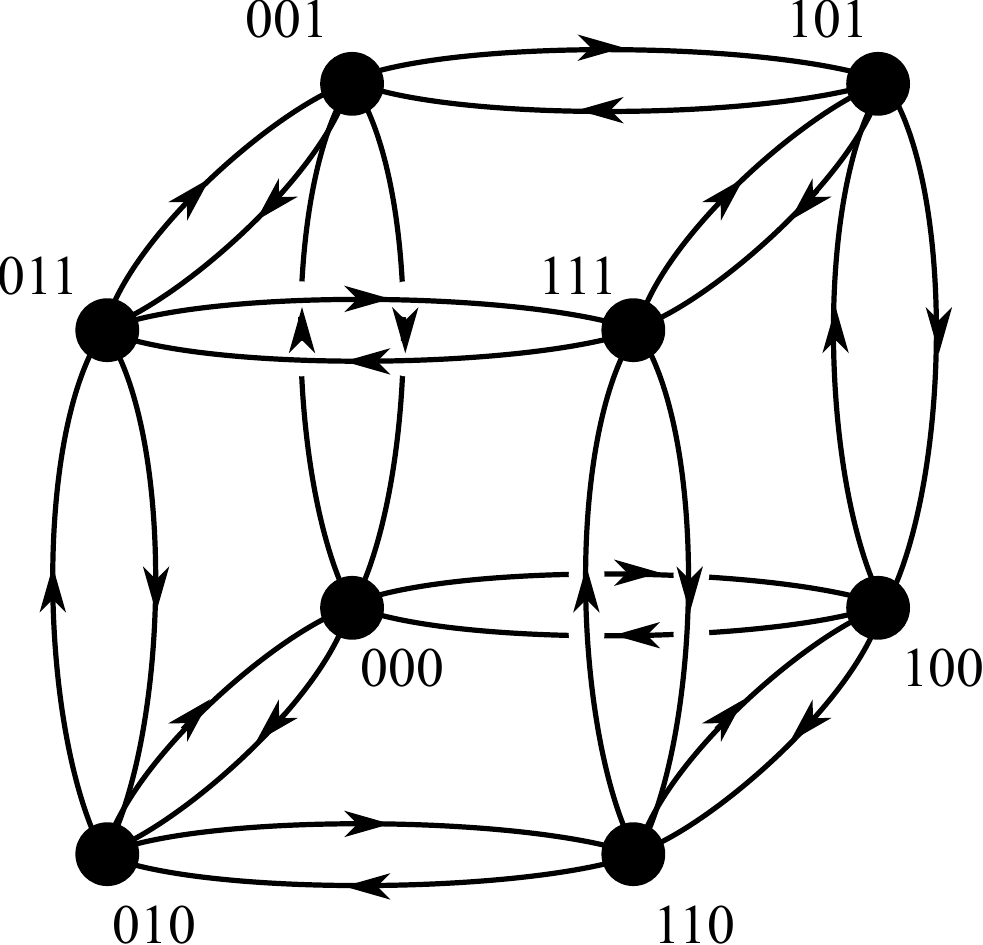}
 \end{center}
 \caption{We assign a state to $P^3$ where faces with the same colours have opposite status. We get the orientation of the 1-skeleton of $C$ shown here.}
 \label{C_arrows:fig}
\end{figure}

\begin{example}
For the 3-coloured $P^3$ we will choose the following state: for every pair of faces with the same colour, assign I to one face and O to the other (the choice of which face gets I and which face gets O will not affect much the result, as we will see). The resulting 1-skeleton of $C$ is then oriented as in Figure \ref{C_arrows:fig}. By Proposition \ref{dichotomy:prop} the homomorphism $f_*$ is not trivial.
\end{example}

\subsection{States and orbits}
Let a right-angled $P\subset \matX^n$ be equipped with a colouring and a state $s$. These determine a diagonal map $f\colon M \to S^1$ as explained above. We now would like to study $f$ and how it depends on $s$. A powerful machinery is already available for this task and is described by Bestvina and Brady in \cite{BB}.

We call $s$ the \emph{initial state}. Recall that $s$ induces a coherent orientation of the edges of the dual cubulation $C$. It also induces a state on every polytope $P_v$ of the tessellation, as follows: every facet $F$ of $P_v$ is dual to an edge $e$ of $C$, and hence inherits a transverse orientation from that of $e$. We assign the status O or I to $F$ according to whether the transverse orientation points outward or inward with respect to $P_v$. It is easy to see that the state induced on $P_v$ is precisely $v(s)$, the result of the action of $v$ on the initial state $s$, as described in Section \ref{states:subsection}.

Summing up, the polyhedron $P_0$ has the initial state $s$, while $P_v$ inherits the state $v(s)$ for each $v\in \matZ_2^c$. The following proposition says that the states that lie in the same orbits produce equivalent diagonal maps.

\begin{prop}
Two states $s,s'$ that lie in the same orbit with respect to the $\matZ_2^c$ action produce two diagonal maps $f,f'\colon M \to S^1$ that are equivalent up to some isometry of $M$, that is there is an isometry $\psi\colon M \to M$ with $f = f'\circ \psi$.
\end{prop}
\begin{proof}
If $s' = w(s)$ for some $w\in \matZ_2^c$, we pick the isometry $\psi\colon M \to M$ that sends $P_v$ to $P_{v+w}$ via the identity map. We get $f=f'\circ \psi$.
\end{proof}

\subsection{Ascending and descending links}
Let a right-angled $P\subset \matX^n$ be equipped with a state $s$. Let $Q$ be a Euclidean polytope combinatorially dual to $P$. When $P=P^3, \ldots, P^8$ of course $Q$ is a Gosset polytope. The state $s$ induces a dual \emph{state} on $Q$, that is the assignment of a \emph{status} I or O to each vertex of $Q$.

If we remove the interiors of the $(n-1)$-octahedral facets from $\partial Q$ (that correspond to the ideal vertices of $P$) we are left with a flag simplicial complex $\dot Q$. This holds because $P$ is right-angled, and hence simple; as a consequence, every face of $Q$ is a simplex, except the $(n-1)$-octahedral facets dual to the ideal vertices of $P$.

Following \cite{BB}, we define the \emph{ascending link} (respectively, \emph{descending link}) as the subcomplex of $\dot Q$ generated by the vertices with status O (respectively, I). Since $\dot Q$ is a flag complex, these subcomplexes are determined by their 1-skeleton.

Let now $P$ be equipped with both a colouring and a state. We get a manifold $M$ and a diagonal map $f\colon M \to S^1$.
For every vertex $v\in \matZ_2^c$ of the dual cubulation $C$, the link of $v$ in $C$
is precisely the simplicial complex $\lk(v) = \dot Q$, and it inherits the state $v(s)$ of $P_v$. The status of a vertex of $\lk(v)$ is I or O according to whether the corresponding oriented edge of $C$ points inward or outward with respect to $v$. The ascending and descending links at $v$ are denoted respectively as $\lk_\uparrow(v)$ and $\lk_\downarrow(v)$, and they are disjoint subcomplexes of $\lk(v)$.

The diagonal map $f\colon M \to S^1$ induces a homomorphism $f_*\colon \pi_1(M) \to \matZ$. We are interested in its kernel $H$.

\begin{teo} \cite[Theorem 4.1]{BB} \label{BB:teo}
The following holds:
\begin{itemize}
\item If $\lk_\uparrow(v)$, $\lk_\downarrow(v)$ are connected for every $v$, then $H$ is finitely generated.
\item If $\lk_\uparrow(v)$, $\lk_\downarrow(v)$ are simply connected for every $v$, $H$ is finitely presented.
\end{itemize}
\end{teo}

\subsection{Legal states}
Following \cite{JNW}, a state $s$ on $P$ is \emph{legal} if the ascending and descending links that it determines on the dual flag simplicial complex $\dot Q$ are both connected. The group $\matZ_2^c$ acts on the set of all states of $P$, and an orbit is \emph{legal} if it consists only of legal states. As noted in \cite{JNW}, Theorem \ref{BB:teo} implies the following. 
\begin{cor} \label{legal:cor}
A legal orbit defines a diagonal map $f\colon M \to S^1$ with finitely generated $H= \ker f_*$. 
\end{cor}

The chase of a legal orbit is the combinatorial game introduced in \cite{JNW}. After introducing the rules of the game, the authors exhibited some legal orbits on two remarkable right-angled polytopes in $\matH^4$, namely the ideal 24-cell and the compact right-angled 120-cell \cite{JNW}, so providing the first algebraically fibering hyperbolic 4-manifolds. Here we play with the right-angled polytopes $P^n$ and find some legal orbits on all of them. More than that, we find some even better kind of orbits in the cases $n=3,7,8$, that we call \emph{1-legal}.

\subsection{1-legal states}
We extend the nomenclature of \cite{JNW} by saying that a state $s$ is \emph{1-legal} if its ascending and descending links are both simply connected. An orbit is \emph{1-legal} if it consists only of 1-legal states. Here is a consequence of Theorem \ref{BB:teo}.

\begin{cor} \label{1-legal:cor}
A 1-legal orbit defines a diagonal map $f\colon M \to S^1$ with finitely presented $H= \ker f_*$.
\end{cor}

\subsection{The Euler characteristic check}
In the following pages we will double count the Euler characteristic of our manifolds as a safety check. If a colouring and a state on $P$ produce a manifold $M$ and a diagonal function $f\colon M \to S^1$, we always have
\begin{equation} \label{chi:eqn}
\chi(M) = \sum_{v \in \matZ_2^c} \big(1-\chi(\lk_{\uparrow}(v)) \big).
\end{equation}
The same formula holds with the descending link $\lk_{\downarrow}(v)$. We say that the integer $1-\chi(\lk_\uparrow(v))$ is the \emph{contribution of $v$} to the Euler characteristic of $M$. Note that a contractible ascending link contributes with zero, while a $k$-sphere contributes with $(-1)^{k+1}$.

We now construct a legal orbit on each individual polytope $P^3, \ldots, P^8$. We have used a code written with Sage to analyse all these cases; both the code and the resulting data are available from \cite{code}.

\subsection{A 1-legal orbit for $P^3$}
In the 3-colouring of $P^3$ the facets are partitioned into three pairs. For every pair, we assign the status O to one facet and I to the other, arbitrarily. The orbit of this state $s$ is independent of this choice and consists precisely of all the $2^3=8$ states that can be constructed in this way.

\begin{figure}
 \begin{center}
  \includegraphics[width = 9 cm]{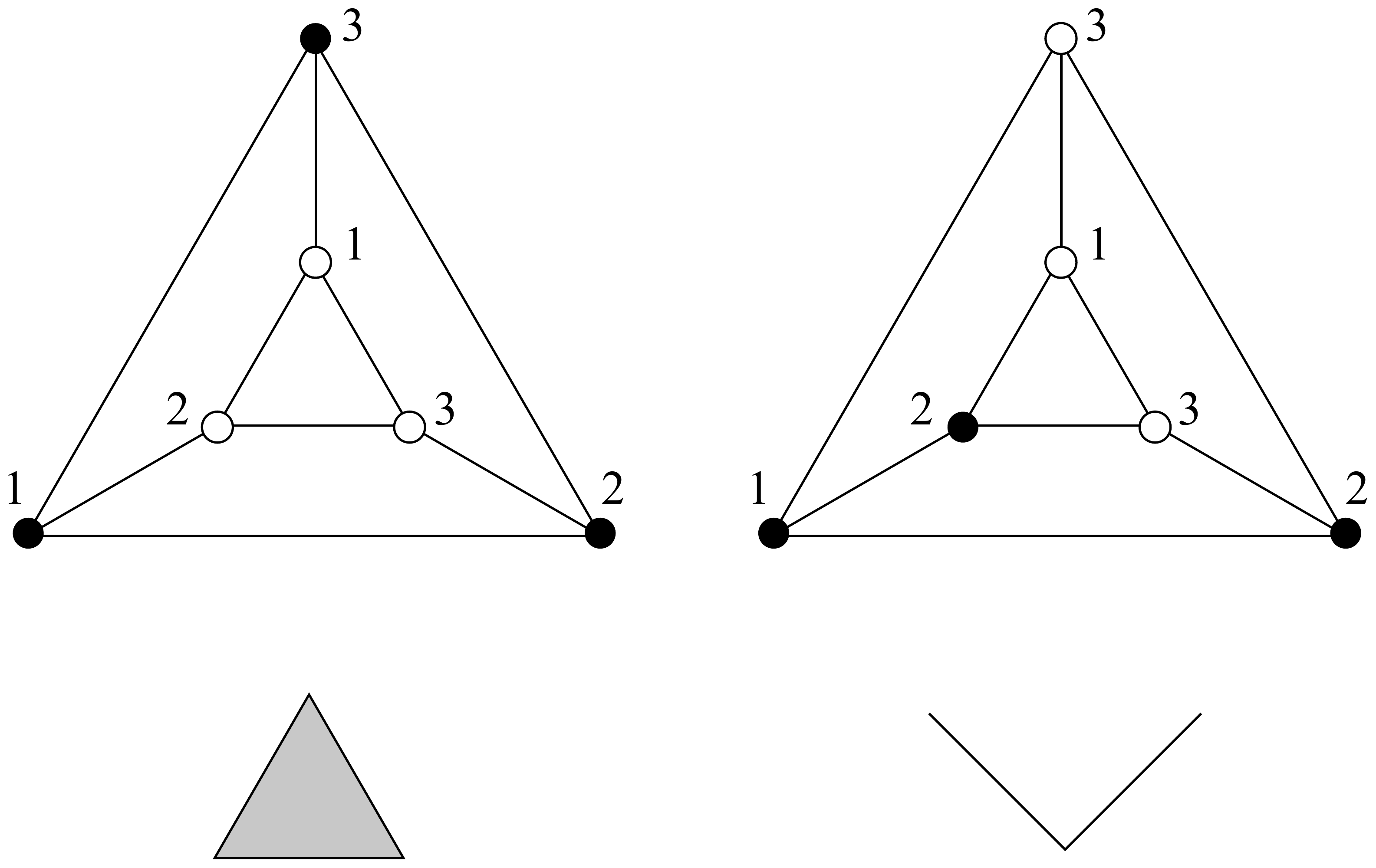}
 \end{center}
 \caption{We exhibit a state by colouring the vertices in black and white, with black (white) corresponding to the status I (O). There are only two states in the orbit of $P^3$ up to isomorphism, and in both cases the ascending and descending links are contractible: they are a triangle and two segments joined along an endpoint.}
 \label{P3_states:fig}
\end{figure}

By direct inspection we find that the 8 states reduce to 2 up to isomorphism, and they are shown in Figure \ref{P3_states:fig}. The ascending and descending links are both either a triangle or two segments connected along an endpoint. In both cases they are contractible. One can in fact verify that the conditions of \cite[Theorem 15]{BM} are satisfied and hence
the diagonal map $f\colon M^3 \to S^1$ can be smoothened to become a fibration (we will not need this here).

The ascending and descending links are simply connected and hence the orbit is 1-legal. By Corollary \ref{1-legal:cor} the kernel $H$ of $f_*\colon \pi_1(M^3) \to \matZ$ is finitely presented: it is the fundamental group of the surface fiber of the fibration $f\colon M^3 \to S^1$.

The formula (\ref{chi:eqn}) holds since $\chi(M^3)=0$ and each contractible link contributes with zero to the sum.

\subsection{A legal orbit for $P^4$}
In the 5-colouring of $P^4$ the facets are partitioned into five pairs. As in the previous case, we assign the statuses O and I arbitrarily to each pair. The orbit consists of all the states that assign distinct statuses to each pair. We get $2^5=32$ states.

\begin{figure}
 \begin{center}
  \includegraphics[width = 12.5 cm]{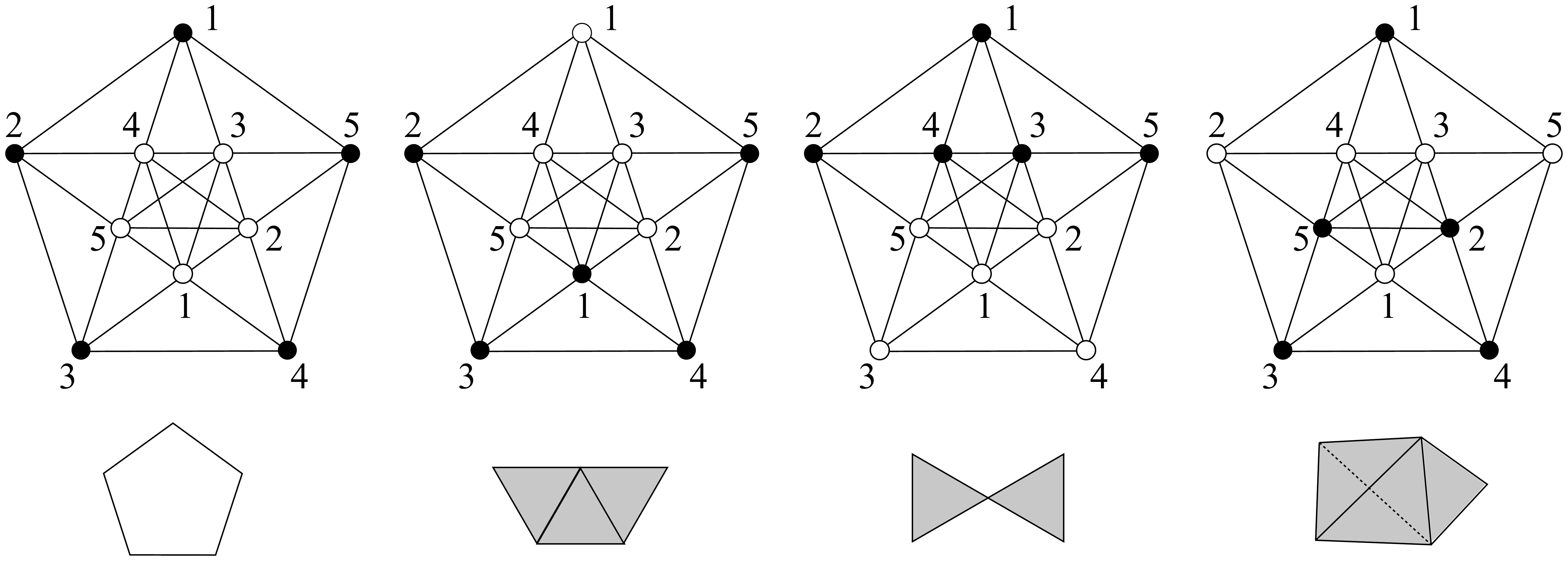}
 \end{center}
 \caption{We exhibit a state by colouring the vertices in black and white, with black (white) corresponding to the status I (O). There are only four states in the orbit of $P^4$ up to isomorphism. We show here the descending link, generated by the black vertices. In the first case we get a circle, while in the other cases we always get a contractible complex made of 3 triangles, 2 triangles, and 1 tetrahedron and 1 triangle respectively. The ascending links are of the same types.}
 \label{P4_states:fig}
\end{figure}

By direct inspection we find that these states reduce to 4 up to isomorphism, depicted in Figure \ref{P4_states:fig}. As shown in the figure, the ascending and descending links are always connected, so the orbit is legal. However, we note that the orbit is not 1-legal, since in the first case both the ascending and descending link are circles. The first case occurs only in 2 of the 32 states.

In fact, one can verify that the ascending and descending links in the first case form a Hopf link in $S^3$, if considered in the boundary of the Gosset polytope, and that the conditions of \cite[Theorem 15]{BM} are satisfied, so the diagonal function $f\colon M^4 \to S^1$ can be smoothened to a circle-valued Morse function with two index-2 critical points. This is the best that we can get in dimension 4, since no fibrations may occur on an even-dimensional hyperbolic manifold (we will not need these facts here, for more details see \cite{BM}).

By Corollary \ref{legal:cor} the kernel $H$ of $f_*\colon \pi_1(M^4) \to \matZ$ is finitely generated. The formula (\ref{chi:eqn}) holds since $\chi(M^4)=2$ and the 2 states of the first kind contribute each with 1, while all the others contribute with 0.

\subsection{A legal orbit for $P^5$}
In the 8-colouring for $P^5$ the facets are partitioned into 8 pairs. As in the previous cases, we assign the statuses O and I arbitrarily to each pair. 
The orbit consists of all the states that assign different statuses to each pair. We get $2^8 = 256$ states. 

\begin{figure}
 \begin{center}
  \includegraphics[width = 10 cm]{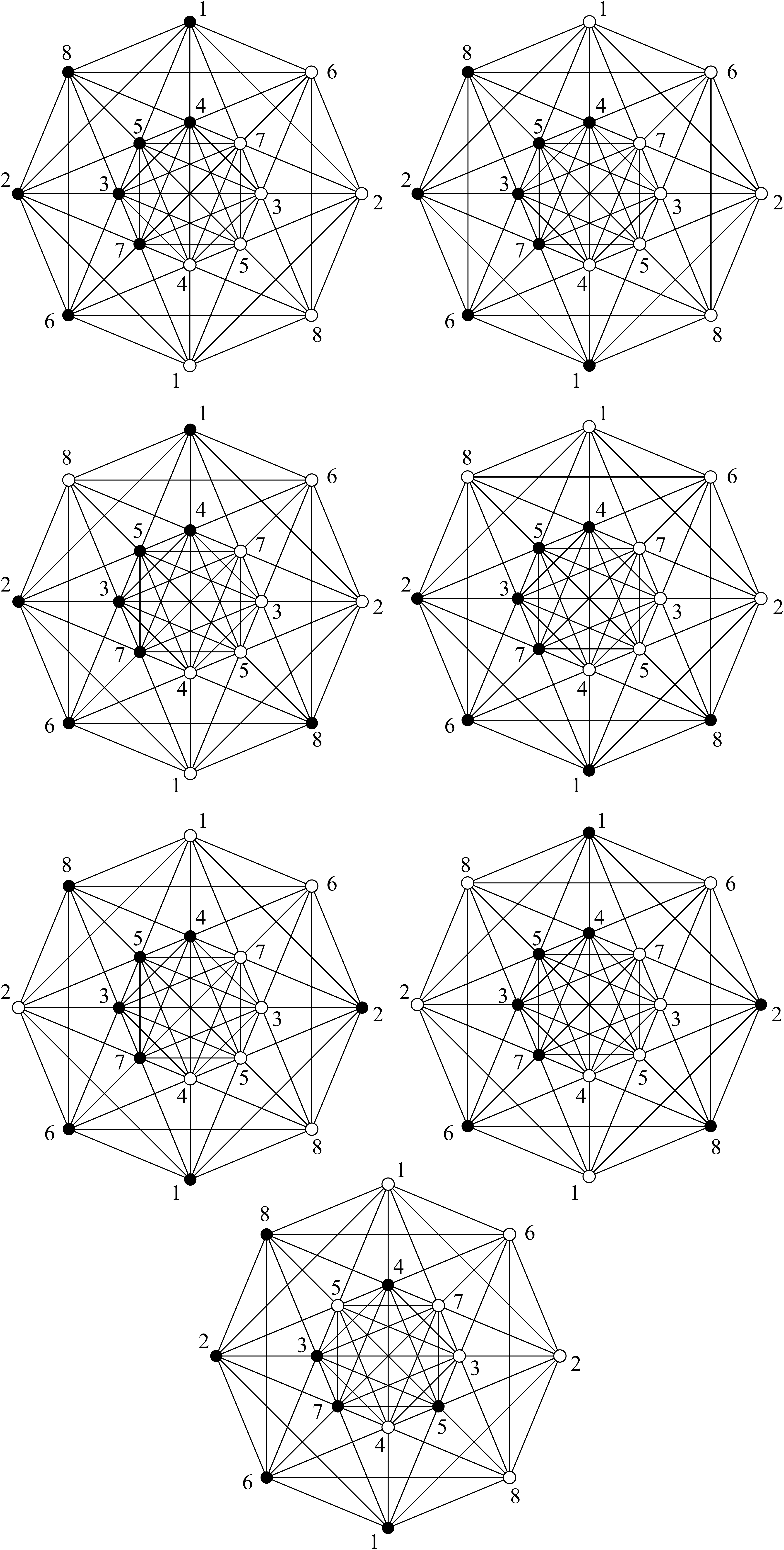}
 \end{center}
 \caption{Every ascending or descending link for $P^5$ is isomorphic to one of the 7 descending links shown here. }
 \label{P5_states:fig}
\end{figure}

Each state produces a pair of ascending and descending links. Since these are flag simplicial complexes, they are determined by their 1-skeleta. Either using our program with Sage or by direct inspection, we find that the resulting 512 graphs reduce to only 7 up to isomorphism. These 7 graphs are those generated by the black vertices in Figure \ref{P5_states:fig}.

\begin{figure}
 \begin{center}
  \includegraphics[width = 2.5 cm]{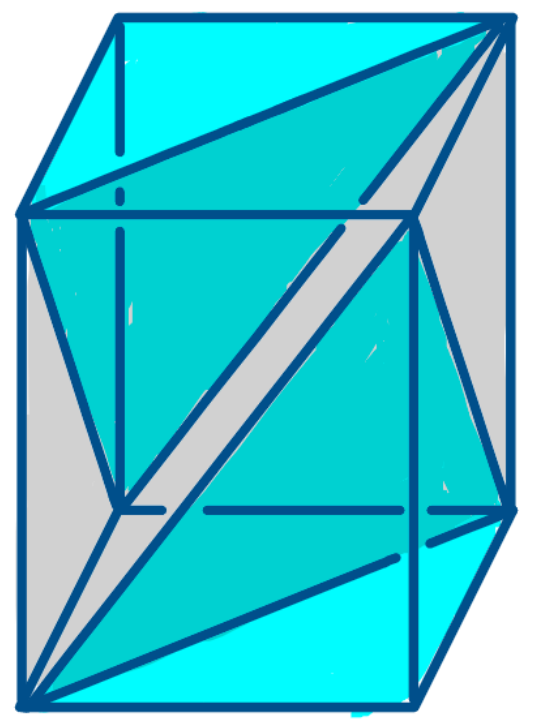} \hspace{1 cm}
  \includegraphics[width = 2.5 cm]{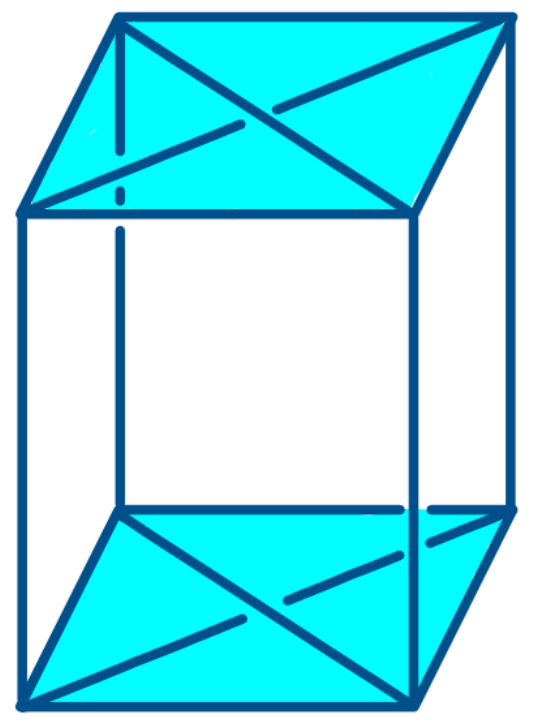}
 \end{center}
 \caption{Two simplicial complexes that collapse to $S^2$ and $\vee_3 S^1$. The first is the boundary of an octahedron with two tetrahedra attached to a pair of opposite faces. The second consists of two tetrahedra and four edges joining them.}
 \label{due_casi:fig}
\end{figure}

We can check by hand (or with our Sage program) that the first 4 graphs in the figure generate a contractible simplicial complex. The remaining three generate a simplicial complex that collapses respectively to $S^2$, a wedge of three circles $\vee_3 S^1$, and $S^3$. The complexes that collapse to $S^2$ and $\vee_3 S^1$ are shown in Figure \ref{due_casi:fig}. The complex that collapses to $S^3$ is actually homeomorphic to $S^3$, and it is the boundary of a 4-octahedron, decomposed into 16 tetrahedra. It corresponds dually to an ideal vertex of $P^5$.

In all the cases the simplcial complex is connected, so the orbit is legal. It is not always simply connected, so the orbit is not 1-legal. By Corollary \ref{legal:cor} the kernel $H$ of $f_*\colon \pi_1(M^5) \to \matZ$ is finitely generated.

The formula (\ref{chi:eqn}) holds since $\chi(M^5)=0$ and using Sage we find that we get 32 occurrences of $S^2$, 8 of $\vee_3 S^1$, and 8 of $S^3$. Their contributions to the Euler characteristic are $32\cdot (-1) + 8\cdot 3 + 8 \cdot 1 = 0$.

\subsection{A legal orbit for $P^6$}
In the 9-colouring for $P^6$ the 27 facets are partitioned into 9 triplets.
As opposite to the previous cases, there does not seem to be a natural choice of a state here. However, a brute computer search shows that there are many legal orbits for $P^6$. For instance, we may take $s$ as the state where the first vertex in each triple listed in Section \ref{M6:subsection} is O and the remaining two are I. By using our Sage program we find that the orbit of this state is legal. By Corollary \ref{legal:cor} the kernel $H$ of $f_*\colon \pi_1(M^6) \to \matZ$ is finitely generated.

As we said, a computer search shows that many initial states $s$ yield legal orbits. We could not find a single 1-legal orbit, but admittedly we have not checked all the possible initial states. 

\subsection{A 1-legal orbit for $P^7$}
In the 14-colouring for $P^7$ the 56 facets are partitioned into 14 quartets. We see $P^7$ as the facet of $P^8$ dual to the vertex $1$ in $4_{21}$. We will define below a state for $P^8$, and this will induce one $s$ for $P^7$ in the obvious way: every facet of $P^7$ inherits the status of the adjacent facet in $P^8$. 

The state $s$ inherited in this way turns out to be \emph{balanced} with respect to the colouring, in the following sense: there is an even number $2m$ of facets sharing the same colour, and precisely half of them $m$ are given the status I, and the other half $m$ the status O. If a state is balanced, then every other state in the orbit is also balanced. The states that we have chosen for $P^3, P^4$, and $P^5$ are balanced: for these polyhedra we have $m=1$ and there was in fact a unique orbit of balanced states. Here $m=2$, so in each quartet two facets receive the status I and two the status O. 

The orbit of $s$ consists of $2^{14}$ states, each contributing with an ascending and descending link. Using Sage we are pleased to discover that the resulting $2^{15} = 32768$ graphs reduce to only 106 up to isomorphism. (This is probably due to the many symmetries of $s$ that are inherited from $P^8$.)

Using Sage we also see that all the simplicial complexes generated by the 106 graphs are simply connected. Therefore the orbit is 1-legal. With Sage we have also checked (\ref{chi:eqn}). All the data can be found in \cite{code}. By Corollary \ref{1-legal:cor} the 
kernel $H$ of $f_*\colon \pi_1(M^7) \to \matZ$ is finitely presented.

\subsection{A 1-legal orbit for $P^8$}
In the 15-colouring for $P^8$ the 240 facets are partitioned into 15 hextets.
How can we find a good initial state $s$ for $P^8$? The numbers are overwhelming: the polytope $P^8$ has 240 facets, so there are $2^{240}$ possible states to choose from. Each orbit consists of $2^{15}$ distinct states, and we would like to find one orbit where \emph{all} these $2^{15}$ states are legal, or even better 1-legal. A brute force computer search is out of reach.

To define a legal state we take inspiration from the 24-cell sitting inside quaternions space, since this has some strong analogies with the Gosset polytope $4_{21}$ sitting in octonions space as we already noticed above. We have already exploited this analogy when we fixed a convenient colouring for $P^8$, and we do it again now to design a convenient state. 

\subsubsection*{A state for the 24-cell}
A nice legal state for the 24-cell was constructed in \cite{BM} as follows. Recall that its 24 vertices are divided into three octets: these are $ \pm 1, \pm i, \pm j, \pm k$,
the elements $\tfrac 12 (\pm 1 \pm i \pm j \pm k) $ with an even number of minus signs, and those with an odd number of minus signs. 

Consider the group $G=\{\pm 1, \pm i\}$ and its action on the 24 vertices by left-multiplication. We can verify easily that each octet is invariant by this action, and is subdivided into two orbits of four elements each. We assign arbitrarily the status I to one orbit, and O to the other. The resulting state $s$ is balanced (see the definition above), and also legal, as it was in fact already observed in \cite{JNW}. The ascending and descending links are each homeomorphic to a $G$-invariant annulus as in Figure \ref{nastri:fig}, so they are connected but not simply connected (the state is not 1-legal). The two $G$-invariant annuli form altogether a banded Hopf link in $S^3$.

\begin{figure}
 \begin{center}
  \includegraphics[width = 8 cm]{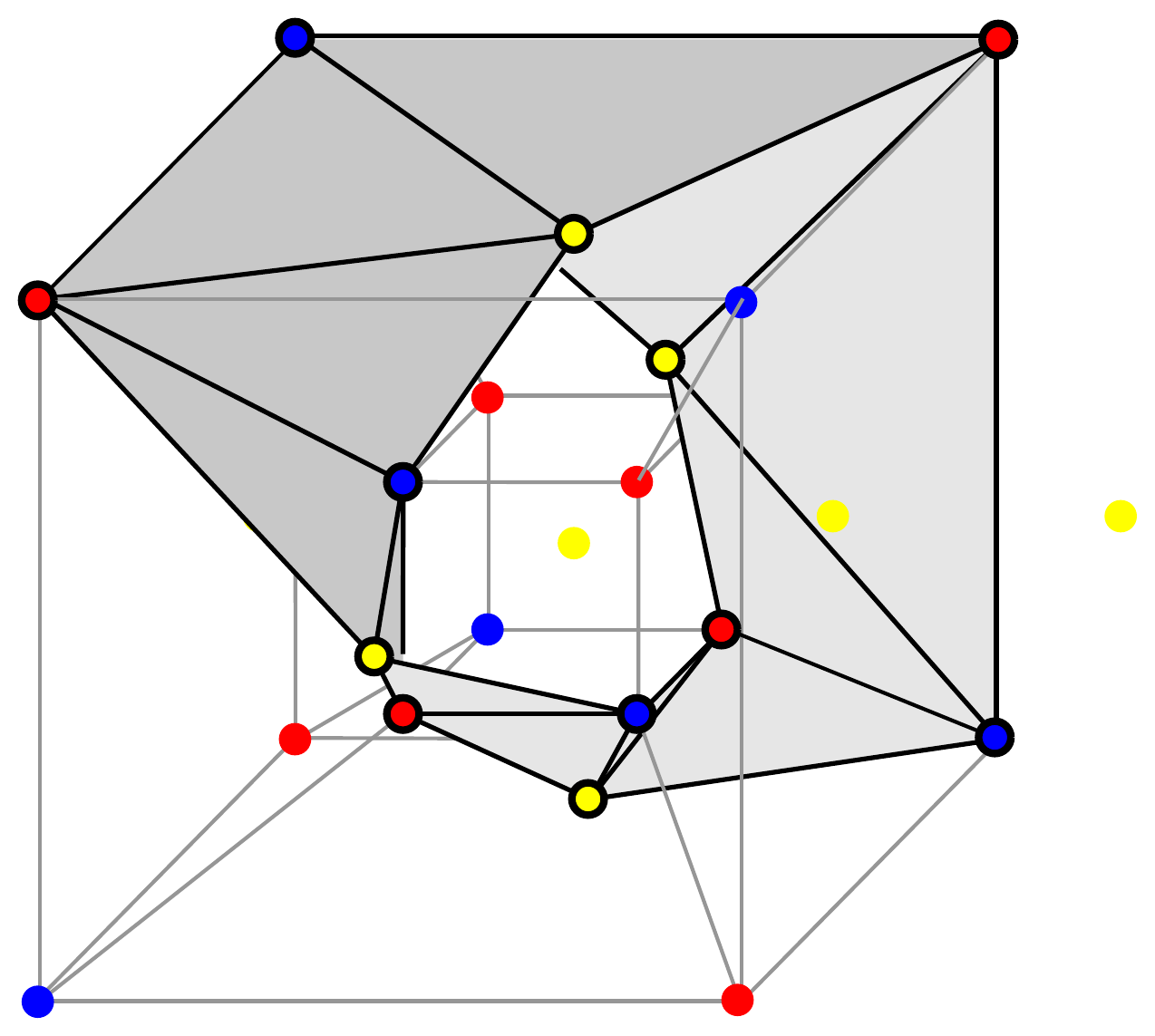}
 \end{center}
 \caption{The ascending and descending link are both an annulus decomposed into 12 triangles, and altogether they form two annuli that collapse onto a Hopf link in $S^3$. The figure (taken from \cite{BM}) shows the vertices of the 24-cell, with their 3-colouring (Blue / Red / Yellow) and their state (the vertices with a O status have an additional black circle). Only some edges of the 24-cell are shown for the sake of clarity: more edges should be added that connect each yellow vertex and its 8 neighbours.}
 \label{nastri:fig}
\end{figure}

The orbit of $s$ along the action of $\matZ_2^3$ consists of the $2^3$ states obtained from $s$ by reversing the I/O status of some octet. The geometry of the 24-cell is so extraordinary that these $2^3$ states are all isomorphic \cite{BM}. In particular, the orbit is legal (but it is not 1-legal). The choice of which orbit is I and which is O inside each octet is irrelevant, since different choices lead to the same orbit.

\subsubsection*{A state for $4_{21}$}
We now try to mimic the above construction for $4_{21}$. The 240 vertices are partitioned into 15 hextets, that is $\pm 1,\pm e_1, \pm e_2, \pm e_3, \pm e_4, \pm e_5, \pm e_6, \pm e_7$, the elements
$\tfrac 12(\pm 1 \pm e_{n} \pm e_{n+1} \pm e_{n+3})$ and
$\tfrac 12(\pm e_{n+2} \pm e_{n+4} \pm e_{n+5} \pm e_{n+6})$ with an even number of minus signs, and those with an odd number of minus signs, with the integer $n$ varying modulo 7.

It is now natural to consider the quaternion group $Q = \{ \pm 1, \pm e_1, \pm e_2, \pm e_4\}$ and its ``action'' on the 240 vertices of $4_{21}$ by left-multiplication. This is the analogue of the group $G=\{\pm 1, \pm i\}$ defined above, roughly because taking quaternions inside octonions looks like taking complex numbers inside quaternions. However, this is not really a group action because octonions are not associative, and hence we may have that $e_1(e_2(x)) \neq (e_1e_2)(x)$.
Therefore some caution is needed.

We already know that the set $S=\{\pm 1, \pm e_1, \ldots, \pm e_7\}$ acts freely and transitively by left-multiplication on every hextet (that is, for every pair of elements in the hextet there is a unique element in $S$ whose left-multiplication sends the first to the second). We pick the following 15 base elements, one inside each hextet:
$$1,\quad 1+e_n+e_{n+1}+e_{n+3},\quad -1+e_n+e_{n+1}+e_{n+3}$$
where $n$ runs modulo 7. We consider inside each hextet the 8 elements obtained by left-multiplying the base element by the elements of $Q$. We assign to these 8 elements the status O, and the status I to the remaining 8 of the hextet.
We have defined a balanced state $s$. The orbit consists of $2^{15}$ balanced states. 

\begin{rem} \label{analogy:rem}
By analogy with the 24-cell, it would be tempting to guess that the $2^{15}$ states in the orbit are all isomorphic, and maybe that the ascending and descending links are homotopic to two copies of $S^3$ linked in $S^7$. This is however impossible by a Euler characteristic argument, due to the fact that the 24-cell has $\chi = 1$ while $\chi(P^8) = 17/2$ is much bigger. In general, one should not push the analogies too far: the situation is intrinsically more complicated here. We will come back to this point below.
\end{rem}

Using our Sage code we have determined the ascending and descending links of each of the $2^{15} = 32768$ states. Note that each graph can have up to 240 vertices, and we have $2^{16} = 65536$ graphs to analyse. 
Luckily, these graphs reduce up to isomorphism to only 185. This is probably due to the extraordinary symmetries of both the colouring and the state that we have chosen for $P^8$. Our Sage program says that each of the 185 simplicial complexes generated by these graphs is connected and simply connected. Therefore the orbit is 1-legal. By Corollary \ref{1-legal:cor} the 
kernel $H$ of $f_*\colon \pi_1(M^8) \to \matZ$ is finitely presented.

\begin{rem}
We also checked (\ref{chi:eqn}). Both sides give (quite reassuringly) the same number $278528$. 
The formula (\ref{chi:eqn}) also explains a fact we alluded to in Remark \ref{analogy:rem}. Since $\chi(P^8) = 17/2$, the average contribution to the Euler characteristic of an ascending link is $17/2$, which is a relatively big number if compared to the Euler characteristic of the other polytopes considered above. Referring to Remark \ref{analogy:rem}, we note that it is certainly impossible that all the ascending links be $S^3$, since their individual contribution would be 1. The ascending links that we find with Sage are indeed quite complicated (they are typically some wedges of spheres of various dimensions $k\geq 2$), much more than those discovered with $P^3,\ldots, P^7$. They can be found in \cite{code}.
\end{rem}

\subsection{The restriction of $f$ to the cusps} \label{f:cusps:subsection}
In our analysis we have briefly mentioned the fact that when $n=3,4$ the chosen orbits satisfy the conditions of \cite[Theorem 15]{BM} and hence $f\colon M^n \to S^1$ can be smoothened to become a perfect smooth circle-valued Morse function (for $n=3$ this is a fibration). 

One may wonder if this is also the case when $n\geq 5$, and indeed this was our hope at the beginning of our study: it would be extremely interesting to find a fibration on an odd-dimensional hyperbolic manifold of dimension 5 or 7. We show that this is not the case, for any possible choice of initial state, a serious obstruction being the restriction of $f$ to the cusps of $M^n$. 

\begin{prop} \label{null:prop}
For $n=5, \ldots, 8$, there is some cusp $X\cong T^{n-1} \times [0, + \infty) \subset M^n$ where the restriction $f|_X$ is homotopic to a constant. This holds for any possible choice of a state for $P^n$.
\end{prop}
\begin{proof}
Let $s$ be any initial state for $P^n$.
In our discussion above we have said that when $n=5,7,8$ there is always some ideal vertex $v$ of $P^n$ whose link $C$ is a $(n-1)$-cube coloured with $2(n-1)$ distinct colours. Let $T\subset M^n$ be a torus section that lies above $C$. The restriction of $f$ to $T$ is determined by the restriction of the state $s$ of $P^n$ to $T$. No matter what the restriction of $s$ looks like, by Corollary \ref{trivial:cor} the restriction of $f$ to $T$ is homotopically trivial, and hence it is so on the cusp $X = T \times [0, + \infty)$ that it bounds.

The case $n=6$ is a bit more involved. When $n=6$, each of the 27 ideal vertices $v$ has a 9-coloured 5-cube link $C$. This implies that there exists exactly one pair of opposite facets sharing the same colour. In each of the 9 triplets of $P^6$, every pair is an opposite pair of facets of this kind, for some ideal vertex $v$ (we get $3\times 9 = 27$ pairs for $27$ ideal vertices). For any choice of a state, there will be one such pair with the same status, because the three statuses on a triple cannot be all distinct. By Proposition \ref{dichotomy:prop} the restriction of $f$ to this cusp is null-homotopic. 
\end{proof}

After writing a first draft of this paper we found a fibration for $M^5$ with a more elaborated construction that overcomes this problem, see \cite{IMM2}.

\subsection{The geometrically infinite coverings}
For every $n=3, \ldots, 8$, the kernel of $f_*\colon \pi_1(M^n) \to \pi_1(S^1) = \matZ$ is a normal subgroup $H\triangleleft \pi_1(M^n)$ of infinite index. We summarise our discoveries. %In the discussion above we have discovered that 
%Corollaries \ref{legal:cor} and \ref{1-legal:cor} imply the following.

\begin{teo}
The normal subgroup $H$ is finitely generated, and also finitely presented when $n=3,7,8$. The limit set of $H$ is the whole sphere $\partial \matH^n = S^{n-1}$.
\end{teo}

The limit set is the whole sphere because $H$ is normal in $\pi_1(M^n)$ and $M^n$ has finite volume. In particular, the hyperbolic $n$-manifold
$$\widetilde M^n = \matH^n/H$$
that covers $M^n$ is geometrically infinite. The dimension $n=4$ was investigated in \cite{BM}. Here we are particularly interested in the dimensions $n=5,\ldots, 8$.

\begin{teo}
When $5\leq n \leq 8$, the hyperbolic manifold $\widetilde M^n$ has infinitely many toric cusps. In particular the Betti number $b_{n-1}(\widetilde M^n) = \infty$ is infinite and $\pi_1(\widetilde M^n)=H$ is not ${\rm F}_{n-1}$.
\end{teo}
\begin{proof}
The restriction of $f$ to some cusp is null-homotopic by Proposition \ref{null:prop}. Therefore the cusp lifts to infinitely many copies of itself in $\widetilde M^n$.
\end{proof}

Recall that a group $H$ is of type ${\rm F}_m$ if there exists a $K(H,1)$ with finite $m$-skeleton \cite{BB}. If $H$ is ${\rm F}_m$, the Betti number $b_m(H)$ is obviously finite. 

\begin{cor}
When $n=7,8$, the fundamental group of the hyperbolic manifold $\widetilde M^n$ is finitely presented but not ${\rm F}_{n-1}$.
\end{cor}

\end{document}